\documentclass[a4paper,12pt]{article} 
\usepackage{graphicx}
\usepackage[english]{babel} 
\usepackage{amsmath,amssymb,amsthm,amsfonts} 
\usepackage{latexsym}
\usepackage[all]{xy}
\usepackage{url}
\usepackage{mathrsfs}
\usepackage[latin1]{inputenc}
\usepackage{mathtools}
\usepackage{hyperref}
\usepackage[usenames]{color}
\usepackage[flushleft]{paralist}

\hypersetup{
 pdftitle={The epsilon constant conjecture for higher dimensional representations},
 pdfauthor={Werner Bley and Alessandro Cobbe}}

\newcommand{\Q}{\mathbb Q}
\newcommand{\N}{\mathbb N}
\newcommand{\Z}{\mathbb Z}
\newcommand{\C}{\mathbb C}

\newcommand{\p}{\mathfrak p}
\newcommand{\fra}{\mathfrak a}
\newcommand{\frp}{\mathfrak p}
\newcommand{\Gal}{\mathrm{Gal}}
\renewcommand{\epsilon}{\varepsilon}
\newcommand{\Kur}{K_{\mathrm{nr}}}
\newcommand{\Nur}{N_{\mathrm{nr}}}

\newcommand{\Zpnr}{\mathbb Z_p^\mathrm{nr}}

\newcommand{\Qpnr}{{\mathbb Q_p^\mathrm{nr}}}

\newcommand{\Tur}{T^\mathrm{nr}}
\newcommand{\Tnr}{\Tur}

\newcommand{\calM}{\mathcal M}
\newcommand{\calN}{\mathcal N}
\newcommand{\mm}{\mathfrak M}

\newcommand{\T}{\mathcal T}
\newcommand{\chinr}{\chi^{\mathrm{nr}}}

\newcommand{\rhonr}{\rho^{\mathrm{nr}}}
\newcommand{\rhoQpnr}{\rho^{\mathrm{nr}}_\Qp}
\newcommand{\BdR}{{B_\mathrm{dR}}}
\newcommand{\Bst}{B_\mathrm{st}}
\newcommand{\Bcris}{B_\mathrm{cris}}
\newcommand{\DdR}{D_\mathrm{dR}}
\newcommand{\Ucris}{U_\mathrm{cris}}

\newcommand{\Dcris}{D_\mathrm{cris}}

\newcommand{\Ind}{\mathrm{Ind}}
\newcommand{\Irr}{\mathrm{Irr}}

\newcommand{\comp}{\mathrm{comp}}
\newcommand{\norm}{\mathcal{N}}
\newcommand{\trace}{\mathcal{T}}
\newcommand{\id}{\mathrm{id}}

\newcommand{\sseq}{\subseteq}
\newcommand{\Qp}{{\Q_p}}
\newcommand{\Zp}{{\Z_p}}

\newcommand{\ZpG}{{\Z_p[G]}}
\newcommand{\QpG}{{\Q_p[G]}}
\newcommand{\tensor}{\otimes}
\newcommand{\tensorQp}{\otimes_\Qp}
\newcommand{\CEP}{C_{EP}^{na}(N/K, V)}

\newcommand{\lra}{\longrightarrow}

\newcommand{\calF}{\mathcal F}
\newcommand{\calL}{\mathcal{L}}
\newcommand{\calZ}{\mathcal{Z}}
\newcommand{\calI}{\mathcal{I}}
\newcommand{\frd}{\mathfrak d}
\newcommand{\frD}{\mathfrak D}
\newcommand{\frf}{\mathfrak f}
\newcommand{\OK}{\mathcal{O}_K}
\newcommand{\ON}{\mathcal{O}_N}
\newcommand{\OL}{\mathcal{O}_L}
\newcommand{\oo}{\mathcal{O}}
\newcommand{\DdRN}{\DdR^N}
\newcommand{\OKG}{{\OK[G]}}
\newcommand{\BdRG}{{\BdR[G]}}

\newcommand{\calO}{\mathcal O}

\newcommand{\Qpc}{\Q_p^c}
\newcommand{\QpcG}{\Q_p^c[G]}

\newcommand{\Hom}{\mathrm{Hom}}

\newcommand{\Nnr}{\calI_{N/K}(\rhonr)}

\newcommand{\tfW}{\tilde f_{3,W}}
\newcommand{\fW}{f_{3,W}}

\newcommand{\fFN}{f_{\calF,N}}

\newcommand{\zz}{\mathcal Z}

\newcommand{\Gl}{\mathrm{Gl}}
\newcommand{\Gm}{\mathbb{G}_m}
\newcommand{\Utwist}{U_{tw}(\rho^{\mathrm{nr}}_\Qp)}
\newcommand{\Ttriv}{T_\mathrm{triv}}
\newcommand{\Vtriv}{V_\mathrm{triv}}
\newcommand{\Nrd}{\mathrm{Nrd}}

\newtheorem{thm}{Theorem}[section]
\newtheorem{lemma}[thm]{Lemma}

\newtheorem{prop}[thm]{Proposition}
\theoremstyle{remark}

\newtheorem{remark}[thm]{Remark}
\newtheorem{remarks}[thm]{Remarks}

\theoremstyle{definition}

\title{The epsilon constant conjecture for higher dimensional unramified twists of $\Z_p^r(1)$}
\author{Werner Bley and Alessandro Cobbe}
\date{}
\begin{document}
\maketitle
\begin{abstract}
    Let $N/K$ be a finite Galois extension of $p$-adic number fields and let $\rhonr \colon G_K \lra \Gl_r(\Zp)$ be an
    $r$-dimensional unramified representation of the absolute Galois group $G_K$ which is the restriction of an unramified
    representation $\rhonr_\Qp \colon G_\Qp \lra \Gl_r(\Zp)$. In this paper we consider the $\Gal(N/K)$-equivariant
    local $\varepsilon$-conjecture for the $p$-adic representation $T = \Z_p^r(1)(\rhonr)$. For example, if
    $A$ is an abelian variety of dimension $r$ defined over $\Qp$ with good ordinary reduction, then the Tate module
    $T = T_p\hat A$ associated to the formal group $\hat A$ of $A$ is a $p$-adic representation of this form. We prove
    the conjecture for all tame extensions $N/K$ and a certain family of weakly and wildly ramified extensions $N/K$.
    This generalizes previous work of Izychev and Venjakob in the tame case and of the authors in the weakly and
    wildly ramified case.
  
\end{abstract}

\tableofcontents

\begin{section}{Introduction}

Let $p$ be a prime and $N/K$ a finite Galois extension of $p$-adic number fields with group $G := \Gal(N/K)$.
We write $G_K$ (resp. $G_N$) for the absolute Galois group of $K$ (resp. $N$) and for each finite extension $E/\Qp$ we let $F_E$ denote the arithmetic Frobenius automorphism. Let $V$ denote a $p$-adic representation of $G_K$ and let $T \sseq V$ be a $G_K$-stable $\Zp$-sublattice such that $V = \Qp \tensor_\Zp T$. 

As in \cite{IV} and \cite{BC2} we write $\CEP$ for the equivariant $\epsilon$-constant conjecture, see for example Conjecture 3.1.1
in \cite{BC2}. For more details  and some remarks on the history of the conjecture we refer the interested reader
to the introduction and Section 3.1 of \cite{BC2}.

In this manuscript we will consider $\CEP$ for higher dimensional unramified twists of $\Z_p^r(1)$ (which should be considered as the Tate module associated with $\Gm^r$). More precisely, by \cite[Prop.~1.6]{Cobbe18}, each matrix $U \in  \Gl_r(\Zp)$ gives rise to an unramified representation of $G_K$ by setting $\rhonr(F_K) := U$. We will be concerned with the module $T=\Z_p^r(1)(\rhonr)$, which by \cite[Prop.~1.11]{Cobbe18} can be considered  as the Tate module of a an $r$-dimensional Lubin-Tate formal group. 

We recall that for $r=1$ and representations $\rhonr$ which are restrictions of unramified extensions $\rhoQpnr : G_\Qp \lra \Z_p^\times$
Izychev and Venjakob in \cite{IV} have proven the validity of $\CEP$ for tame extensions $N/K$. The main result of \cite[Thm.~1]{BC2} shows 
that $\CEP$ holds for certain weakly and wildly ramified finite abelian extensions $N/K$. In this context we recall that $N/K$ is weakly ramified if the second ramification group in lower numbering is trivial. Generalizing these results we will show:

\begin{thm}\label{intro thm tame}
Let $N/K$ be a tame extension of $p$-adic number fields and let
\[\rhoQpnr \colon G_\Qp \lra \Gl_r(\Zp)\]
be an unramified representation of $G_\Qp$. Let $\rhonr$ denote the restriction of $\rhoQpnr$ to $G_K$. Then $\CEP$ is true for $N/K$ and $V = \Q_p^r(1)(\rhonr)$, if $\det(\rhonr(F_N) - 1) \ne 0$. 
\end{thm}

\begin{remarks}\label{intro rem 1}
\begin{enumerate}[(a)]
\item The condition $\det(\rhonr(F_N) - 1) \ne 0$ holds, if and only if $H^2(N, T)$ is finite (see Section \ref{sec cohomology}). 
It is also equivalent to $\left( \Z_p^r(\rhonr) \right)^{G_N} = 0$.
\item If $r=1$, then $\det(\rhonr(F_N) - 1) = 0$ if and only if $\rhonr |_{G_N} = 1$. If $r > 1$, then there are 'mixed' cases where
  both $\rhonr |_{G_N} \ne 1$ and  $\det(\rhonr(F_N) - 1) = 0$, see, e.g., \cite[Example 3.18]{Cobbe18}.
\item If $\rhonr |_{G_N}=1$, then twisting commutes with taking $G_N$-cohomology, so that
    $\CEP$ can be proved quite easily relying on the fact that the conjecture is known in the untwisted case by \cite{Breuning04}.
\end{enumerate}
\end{remarks}

In the weakly ramified setting we will prove the following theorem.

\begin{thm}\label{intro thm weak}
{Let $p$ be an odd prime.} Let $K/\Q_p$ be the unramified extension of degree $m$ and let $N/K$ be a weakly and wildly ramified finite
  abelian extension with cyclic ramification group. Let $d$ denote the inertia degree of $N/K$, let
    $\tilde d$ denote the order of $\rhonr(F_N) \text{ mod } p$ in  $\Gl_r(\Zp/p\Zp)$ and assume that $m$ and $d$ are relatively prime. Let
  \[
    \rhoQpnr \colon G_\Qp \lra \Gl_r(\Zp)
  \]
  be an unramified representation of $G_\Qp$ and let $\rhonr$ denote the restriction of $\rhoQpnr$ to $G_K$.
  Assume that $\det(\rhonr(F_N) - 1) \ne 0$ (Hypothesis (F)) and, in addition, that
one of the following three conditions holds:
\begin{enumerate}
\item[(a)] $\rhonr(F_N)-1$ is invertible modulo $p$ (Hypothesis (I)); 
\item[(b)] $\rhonr(F_N)\equiv 1\pmod{p}$ (Hypothesis (T));
\item[(c)] $\gcd(\tilde d,m)=1$ and $\det(\rhonr(F_N)^{\tilde d} - 1) \ne 0$.
\end{enumerate}
Then $\CEP$ is true for $N/K$ and $V = \Q_p^r(1)(\rhonr)$. 
\end{thm}

\begin{remarks} 
\begin{enumerate}
	\item[(a)] Note that the conditions (a) and (b) concerning the reduction of $\rhonr(F_N)$ modulo $p$
  generalize the cases $\omega=0$ and $\omega>0$, which were studied separately in \cite{BC2}, and which exhaust all the possible cases when $r=1$. In the higher dimensional setting of the present paper, however, this is not true, even under the assumption $\det(\rhonr(F_N) - 1) \ne 0$. To deal with the remaining cases, our strategy of proof is to replace the field $N$ by its unramified extension of degree $\tilde d$ and to use functoriality with respect to change of fields (see Prop.~\ref{functoriality}). For technical reasons this forces us to require hypothesis (c).
	\item[(b)] By \cite[Lem.~1.1]{Cobbe18} we know that $\tilde d$ is a divisor of $p^st$ with $s = (r-1)r/2$ and $t = \prod_{i=1}^r(p^i-1)$.
\end{enumerate}
\end{remarks}

In a more geometrical setting, if $A/\Q_p$ is an abelian variety of dimension $r$ with good ordinary reduction, then by \cite[Prop.~1.12]{Cobbe18} the Tate module of the associated formal group $\hat{A}$ is isomorphic to $\Z_p^r(1)(\rhoQpnr)$ for an appropriate choice of $\rhoQpnr$. Here it is worth to remark that the converse is not true, i.e. not every module $\Z_p^r(1)(\rhoQpnr)$ comes from an abelian variety with good ordinary reduction. In this setting, by a result of Mazur \cite[Cor.~4.38]{Mazur72}, we know that  $\det(\rhonr(F_L) - 1) \ne 0$ is automatically satisfied for any finite extension $L/\Qp$, see Lemma \ref{Mazur lemma}.

\begin{thm}\label{intro thm tame geom}
Let $N/K$ be a tame extension of $p$-adic number fields and let $A/\Q_p$ be an $r$-dimensional abelian variety with good ordinary reduction. Let $\rhonr_\Qp$ be the unramified representation induced by the Tate module $T_p\hat A$ of the formal group $\hat A$ of $A$ and let $\rhonr$ be the restriction of $\rhonr_\Qp$ to $G_K$. Then $\CEP$ is true for $V = \Q_p\otimes_{\Z_p}T_p\hat A$.
\end{thm}

\begin{thm}\label{intro thm weak geom}
  Let $p$ be an odd prime and let $A/\Q_p$ be an $r$-dimensional abelian variety with good ordinary reduction.
  Let $K/\Qp$ be the unramified extension of degree $m$ and let $N/K$ be a weakly and wildly ramified finite
    abelian extension with cyclic ramification group. Let $\rhonr_\Qp$ be the unramified representation induced by
   the Tate module $T_p\hat A$ of the formal group $\hat A$ of $A$ and let $\rhonr$ be the restriction of $\rhonr_\Qp$ to $G_K$.
    Let $d$ denote the inertia degree of $N/K$ and let
    $\tilde d$ denote the order of $\rhonr(F_N) \text{ mod } p$ in  $\Gl_r(\Zp/p\Zp)$.
    Assume that $m$ and $d$ are relatively prime, and, in addition,  that one of the following conditions hold:
\begin{enumerate}
\item[(a)] $\rhonr(F_N)-1$ is invertible modulo $p$ (Hypothesis (I));
\item[(b)] $\rhonr(F_N)\equiv 1\pmod{p}$ (Hypothesis (T));
\item[(c)] $(m, \tilde d) = 1$.
\end{enumerate}
Then $\CEP$ is true for $V = \Q_p\otimes_{\Z_p}T_p\hat A$.
\end{thm}

To conclude this introduction we reference forthcoming work of Nickel \cite{NickelPreprint}
and a forthcoming joint paper of Burns and Nickel \cite{BuNi}
where an Iwasawa theoretic approach
to $\CEP$ is developed. In a little more detail, Nickel formulates an Iwasawa theoretic analogue of $C_{EP}^{na}(N/K, \Qp(1))$,
call it $C_{EP}^{na}(N_\infty/K, \Qp(1))$ for the purpose of this introduction,
for the extension $N_\infty/K$ where $N_\infty/N$ is the unramified
$\Zp$-extension of $N$. Then, in a second paper, Burns and Nickel show that $C_{EP}^{na}(N_\infty/K, \Qp(1))$
holds if and only if $C_{EP}^{na}(E/F, \Qp(1))$ holds
for all finite Galois extensions $E/F$ such that $K \sseq F \sseq E \sseq N_\infty$. Furthermore they prove a certain twist
invariance of the conjecture. If $\chinr_\Qp$ is a one-dimensional unramified character, they show that
$C_{EP}^{na}(N_\infty N'/K, \Qp(1))$ holds if and only if $C_{EP}^{na}(E/F, \Qp(1)(\chinr))$ holds
for all finite Galois extensions $E/F$ such that $K \sseq F \sseq E \sseq N_\infty N'$
where $N'/N$ is a certain unramified extension of degree dividing $p-1$. It will be very interesting to see how this
Iwasawa theoretic approach will carry over
to the higher dimensional case.

{\bf Notations: }We will mostly rely on the notation of \cite{BC2} and \cite{Cobbe18}. In particular, $N/K$ will always denote a finite Galois extension,
  $N_0$ will be the completion of the maximal unramified extension $N^\mathrm{nr}$ of $N$ and $\widehat{N_0^\times}$ the $p$-completion of $N_0^\times$. Let $N_1$ be the maximal unramified subextension of $N/\Qp$.
  We will denote by $e_{N/K}$ and $d_{N/K}$ the ramification index and the inertia degree of $N/K$, $\oo_N$ will be the ring of integers of $N$ and $U_N$
  will be its group of units. We also set $\Lambda_N=\prod_r \widehat{N_0^\times}(\rhonr)$, $\Upsilon_N=\prod_r \widehat{U_{N_0}}(\rhonr)$ and $\zz=\Z_p^r(\rhonr)$
  and we will mostly use an additive notation for the (twisted) action of the absolute Galois group $G_N$. The elements fixed by the action of $G_N$
  will be denoted by $\Lambda_N^{G_N}$, $\Upsilon_N^{G_N}$ and $\zz^{G_N}$, respectively.

Let $\varphi$ be the absolute Frobenius automorphism, let $F_N$ be the Frobenius automorphism of $N$ and let $F=F_K$ be the Frobenius of $K$.

For an $r$-dimensional formal group $\calF$, we denote by $\calF(\p_N^{(r)})$ the group structure on $\prod_r\p_N$ induced by $\calF$.

For any ring $R$ we denote by $M_r(R)$ the ring of $r\times r$ matrices with coefficients in $R$ and by $\Gl_r(R)$ the group of invertible matrices. A unity matrix will always be denoted simply by $1$. We also write $Z(R)$ for the centre of $R$.

If $\Lambda$ and $\Sigma$ are unital rings and $\Lambda \lra \Sigma$ a ring homomorphism, then we write
  $K_0(\Lambda, \Sigma)$ for the relative algebraic $K$-group defined by Swan \cite[p.~215]{Swan}. If $\Sigma = L[G]$ for
  a finite group $G$ and a field extension $L/\Qp$ we write $\Nrd_\Sigma \colon K_1(\Sigma) \lra Z(\Sigma)^\times$ for the map
  on $K_1$ induced by the reduced norm map. We will only be concerned with cases where $\Nrd_\Sigma$ is an isomorphism.
  In this case we set $\hat\partial^1_{\Lambda, \Sigma} := \partial^1_{\Lambda, \Sigma} \circ \Nrd_\Sigma^{-1}$ where
  $\partial^1_{\Lambda, \Sigma} \colon K_1(\Sigma) \lra K_0(\Lambda, \Sigma)$ is the canonical map. If there is no danger of
  confusion we will often abbreviate $\hat\partial^1_{\Lambda, \Sigma}$ to $\hat\partial^1$.

For any $\Zp$-module $X$ and any ring extension $R/\Zp$ we set $X_R := R \tensor_\Zp X$.

\begin{subsection}{Plan of the manuscript}

We will start recalling some results on the cohomology of $\Z_p^r(1)$ which are proved in \cite{Cobbe18}. We will also formulate a finiteness hypothesis (F), which we will assume throughout the paper, and we will show some basic consequences of (F). After a short digression on the formal logarithm and exponential function in higher dimension in Section \ref{sectlog}, we can start our study of the conjecture $\CEP$.

As in \cite{BC2}, which was motivated by the work in \cite{IV},  we define an element

\begin{equation}\label{defRNK}
\begin{split}  R_{N/K}&= C_{N/K} + \Ucris+rm\hat\partial^1_{\ZpG, \BdR[G]}(t) - m \Utwist \\
&\qquad\qquad- r U_{N/K} + \hat\partial^1_{\ZpG, \BdR[G]}(\epsilon_D(N/K,V))
\end{split}\end{equation}
in the relative algebraic $K$-group  $K_0(\ZpG,\QpG)$. The conjecture $\CEP$ is then equivalent to the vanishing of 
$R_{N/K}$.

Actually the element $R_{N/K}$ as defined in (\ref{defRNK}) differs from \cite[(17)]{BC2} by the term
  $m \Utwist$. This new term emerges from the computation of the cohomological term $C_{N/K}$,
  which was slightly incorrect
in \cite{BC2}, and has to be compensated in the definition of $R_{N/K}$. For more details on this issue we refer the
reader to Remark \ref{Remark Error}.

We will explicitly compute the terms
$C_{N/K}$, $\Ucris$ and  $\hat\partial^1_{\ZpG, \BdR[G]}(\epsilon_D(N/K,V))$ in the definition of $R_{N/K}$ and then use these results to 
prove $\CEP$ when $N/K$ is tame ({Theorem \ref{intro thm tame}}) and, under some additional hypotheses, 
also when $N/K$ is weakly and wildly ramified ({Theorem \ref{intro thm weak}}). 
This generalizes previous work for $r = 1$ of Izychev and Venjakob in \cite{IV} in the tame case and the authors in \cite{BC2} in the weakly ramified case.

\end{subsection}

\end{section}

\begin{section}{\texorpdfstring{The cohomology of $\Z_p^r(1)(\rhonr)$}{The cohomology of Zpr(1)(rho nr)}}\label{sec cohomology}

  Let $u \in \Gl_r(\Zp)$ and let $\rhonr = \rho_u \colon G_\Qp \lra \Gl_r(\Zp)$ denote the unramified representation attached to
  $u$ by \cite[Prop.~1.6]{Cobbe18}. By \cite[Sec.~13.3]{Hazewinkel78} there is a unique $r$-dimensional Lubin-Tate formal group
  $\calF = \calF_{pu^{-1}}$ attached to the parameter $pu^{-1}$. As in \cite[Prop.~1.10]{Cobbe18} we can construct an
  isomorphism $\theta \colon \calF \lra \Gm^r$ defined over the completion $\overline{\Qpnr}$ of $\Qpnr$ such that
  \begin{equation}\label{theta iso}
    \theta(X) = \epsilon^{-1} X + \ldots \text{ and } \theta^\varphi \circ \theta^{-1} = u^{-1}
  \end{equation}
  where $\epsilon \in \Gl_r(\overline{\Zpnr})$ has the defining property $\varphi(\epsilon^{-1})\epsilon = u^{-1}$.
  In the following we set $T := \Z_p^r(1)(\rhonr)$ and for future reference we recall that $T$ is isomorphic to the
  $p$-adic Tate module $T_p\calF$ of $\calF$ by \cite[Prop.~1.11]{Cobbe18}.

  Let $N/\Qp$ be a finite field extension and let $N_0 = \overline{N^{nr}}$ denote the completion of the maximal unramified extension of $N$.
  Following \cite{Cobbe18} we define
\[
\Lambda_N := \prod_r \widehat{N_0^\times}(\rhonr), \quad \calZ := \Z_p^r(\rhonr).
\]
Then, by \cite[Cor.~3.16]{Cobbe18}, we have
\begin{eqnarray*}
H^1(N, T) &\cong& \Lambda_N^{G_N} \cong \calF(\frp_N^{(r)}) \times \calZ^{G_N}, \\
H^2(N, T) &\cong& \calZ / (F_N - 1)\calZ, \\
H^i(N, T) &=& 0 \text{ for } i \ne 1,2.
\end{eqnarray*}

\begin{remark}\label{cohomology iso}
  We point out that the above isomorphisms are induced by the explicit representative $C^\bullet_{N, \calF}$ of $R\Gamma(N, T)$ constructed in
  \cite[Thm.~3.15]{Cobbe18}. In the formulation of $\CEP$, however, we will use the identification of the cohomology modules resulting from
  the use of continuous cochain cohomology. We will address this problem in Section \ref{Preliminary results}.
  
\end{remark}

For each finite field extension $N/\Qp$ we set
\[
U_N := \rho_u(F_N) = \rhonr(F_N) = u^{d_N}
\]
and in the sequel always assume the following finiteness hypothesis.

\bigskip

\noindent{\bf Hypothesis (F):} $\det(U_N - 1) \ne 0$.

\bigskip

\noindent This hypothesis clearly implies (and, in fact, is equivalent to)
\begin{eqnarray*}
H^1(N, T) &\cong& \calF(\frp_N^{(r)}) \\
H^2(N, T) &\cong& \calZ / (F_N - 1)\calZ = \Z_p^r / (U_N-1)\Z_p^r \text{ is finite}.
\end{eqnarray*}
The elementary divisor theorem immediately implies
\[
  \# (\calZ / (F_N - 1)\calZ) = p^\omega \text{ with } \omega := v_p(\det(U_N - 1)),
\]
where $v_p$ denotes the normalized $p$-adic valuation.

We first study the case when $N/K$ is tame. 

\begin{prop}\label{tame cohomology}
Let $N/K$ be a finite Galois extension with Galois group $G = \Gal(N/K)$.   
Assume that Hypothesis (F) holds.  
If $N/K$ is tame, then both $H^1(N, T)$ and  $H^2(N, T)$
are $G$-cohomologically trivial.
\end{prop}
\begin{proof} 
By \cite[Thm.~3.3 and Lemma 2.2]{Cobbe18} it suffices to show that $\calZ / (F_N - 1)\calZ = \Z_p^r / (U_N - 1)\Z_p^r$
is cohomologically trivial.

We set $M := \calZ / (F_N - 1)\calZ$ and write $I = I_{N/K}$ for the inertia group.
By \cite[Lemma 2.3]{EllerbrockNickel18} it suffices to show that $H^i(G/I, M^I) = 0$ and   $H^i(I, M) = 0$ for all $i \in \Z$.
Since $M$ is a (finite) $p$-group and $p \nmid \#I$ we get $H^i(I, M) = 0$. Hence it suffices to show that
 $H^i(G/I, M^I) = 0$ for all $i \in \Z$. Since $G/I$ is cyclic and $M$ finite, a standard Herbrand quotient argument shows that 
it is then enough to prove that $H^{-1}(G/I, M^I) = 0$.
Note that $M^I = M$. The long exact cohomology sequence attached to the short exact sequence 
\[
0 \lra \calZ \xrightarrow{F_N - 1} \calZ \lra M \lra 0
\]
of $G/I$-modules yields the exact sequence
\[
H^{-1}(G/I, \calZ ) \lra H^{-1}(G/I,  M) \lra H^{0}(G/I, \calZ ).
\]
With $F_N = F_K^{d_{N/K}}$ one has $U_N = U_K^{d_{N/K}}$ and
\[
1 -U_N = ( 1 + U_K + \ldots + U_K^{d_{N/K}-1})(1 - U_K).
\]
Since $U_N - 1$ is invertible, the same is true for   $1 - U_K$, hence $\calZ^{G/I} = 0$ and $H^{0}(G/I, \calZ ) = 0$.
To show that $H^{-1}(G/I, \calZ ) = 0$ we note that the above identity also implies  that $1 + U_K + \ldots + U_K^{d_{N/K}-1}$ is invertible, 
and hence the kernel of the norm map is trivial. Consequently, $H^{-1}(G/I, \calZ ) = 0$.
\end{proof}

Because of Proposition \ref{tame cohomology} the tame case is much more accessible to proofs of  conjecture $\CEP$ than the wild case.
Conversely, the following lemma shows that in the generic wild case the cohomology modules are not cohomologically trivial.

\begin{lemma}\label{ct coh}
  Assume that Hypothesis (F) holds. Then the following are equivalent:
\begin{enumerate}
\item[(i)] $H^2(N,T)$ is  trivial.
\item[(ii)] $U_N - 1 \in \Gl_r(\Zp)$.
\end{enumerate}
If $N/K$ is wildly ramified, then this is also equivalent to
\begin{enumerate}
\item[(iii)] $H^1(N,T)$ is cohomologically trivial.
\item[(iv)] $H^2(N,T)$ is cohomologically trivial.
\end{enumerate}
\end{lemma}

\begin{proof}
  The equivalence of (i) and (ii) is clear. The equivalence of (iii) and (iv) follows from 
  \cite[Thm.~3.3 and Lemma 2.2]{Cobbe18}. To see the equivalence of (i) and (iv) in the wildly ramified case
  it suffices to note that $I$ acts trivially on
$M := H^2(N,T) =  \calZ / (F_N - 1)\calZ$. If $P$ denotes a subgroup of $I$ of order $p$, then one obviously has
\[
H^1(P, M) = \Hom(P, M) = 0 \iff M = 0.
\]

\end{proof}
\end{section}

\begin{section}{Formal logarithm and exponential function in higher dimensions}\label{sectlog}
  In this section we prove some results which are probably well-known, but for which we could not find a precise reference in the literature.
  Throughout this subsection we let $L$ be a finite extension of $\Qp$ and $v_L$ the normalized valuation of $L$.
  
  The statement and the proof of the following lemma generalize \cite[IV.1, Prop.~1]{Froehlich68} to a higher dimensional setting.
  We set $X := (X_1, \ldots, X_r)$.

\begin{lemma}\label{logar}
  Let $\calF$ be an $r$-dimensional commutative formal group defined over $\Z_p$. Then there exists a unique isomorphism
  $\log_\calF:\calF\to \mathbb{G}_a^r$ defined over $\Q_p$, so that the Jacobian $J_{\log_\calF}(X)$ satisfies $J_{\log_\calF}(0)=1$.
  Furthermore, $J_{\log_\calF}(X)\in M_r(\Z_p[[X]])$ and $\log_\calF(x)$ converges for all $x=(x_1,\dots,x_r) \in L^{(r)}$ satisfying
  $\min\{v_L(x_1),\dots,v_L(x_r)\}>0$.
\end{lemma}

\begin{proof}
  By \cite[II.2, Thm.~1, Cor.~1]{Froehlich68} there exists an isomorphism $g:\calF\to \mathbb{G}_a^r$ defined over $\Q_p$.
  It is then clear that the Jacobian $J_g(0)$ is an invertible matrix. We also note that $J_g(0)^{-1}X$ defines an isomorphism $g_1:\mathbb{G}_a^r\to \mathbb{G}_a^r$. Thus the composition $\log_\calF=g_1\circ g:\calF\to \mathbb{G}_a^r$ is an isomorphism satisfying
  our normalization $J_{\log_\calF}(0)=J_{g_1}(g(0))J_g(0)=J_g(0)^{-1}J_g(0)=1$.

  To prove uniqueness we assume that $f:\calF\to \mathbb{G}_a^r$ is another isomorphism with $J_f(0)=1$.
  Then
  \[
    J_{\log_\calF\circ f^{-1}}(0)=J_{\log_\calF\circ f^{-1}}(f(0))=J_{\log_\calF}(0)J_{f^{-1}}(f(0))=J_{\log_\calF}(0)J_{f}(0)^{-1}=1.
  \]
  It is easy to see that the isomorphisms $\mathbb{G}_a^r\to \mathbb{G}_a^r$ over $\Z_p$ are in one to one
  correspondence with the invertible matrices in $M_r(\Z_p)$. Hence we deduce that $\log_\calF\circ f^{-1}$ is the identity map,
  i.e., $f=\log_\calF$. 

To show that $J_{\log_\calF}(X)\in M_r(\Z_p[[X]])$, we write
\[
  \log_\calF(\calF(X,Y))=\log_\calF(X)+\log_\calF(Y).
\]
We view both sides as formal series in the variables $Y$, calculate the Jacobians and evaluate at $Y=0$:
\[
  J_{\log_\calF}(\calF(X,0))J_{\calF(X,\cdot)}(0)=0+J_{\log_\calF}(0).
\]
As a consequence we obtain
\[
  J_{\log_\calF}(X)J_{\calF(X,\cdot)}(0)=1.
\]
We let $\fra$ denote the ideal of $\Zp[[X]]$ which is generated by $X_1, \ldots, X_r$ and note that $p\Zp[[X]] + \fra$ is the maximal
ideal of the local ring $\Zp[[X]]$. By the axioms of formal groups it follows that
$J_{\calF(X,\cdot)}(0) = 1 + M$ with a matrix $M\in M_r(\Z_p[[X]])$ with coefficients in $\fra$. Hence
$\det(J_{\calF(X,\cdot)}(0)) \equiv 1 \pmod{\fra}$ and we deduce that $\det(J_{\calF(X,\cdot)}(0))$ is a unit in $\Zp[[X]]$.
It follows that $J_{\calF(X,\cdot)}(0)$ is invertible in $M_r(\Zp[[X]])$, so that its inverse $J_{\log_\calF}(X)$ has integral coefficients.

Hence a general term of any component $\log_{\calF,i}$ of $\log_\calF$ is of the form $\frac{a}{m}\prod_{i=1}^r X_i^{n_i}$,
with $m=\gcd(n_1,\dots,n_r)$ and $a\in\Z_p$. If we set $n=\sum_{i=1}^r n_i$, then
\[
  v\left(\frac{a}{m}\prod_{i=1}^r x_i^{n_i}\right)\geq\sum_{i=1}^r n_iv(x_i)-v(m) \ge n\min\{v_L(x_1),\dots,v_L(x_r)\}-(\log_p n)v_L(p).
\]
This last expression tends to infinity when the total degree $n$ tends to infinity.
\end{proof}

As usual we write $\exp_\calF$ for the inverse of $\log_\calF$. To obtain information on the convergence of $\exp_\calF$ 
we will need the following lemma whose proof is inspired by the proof of \cite[Lemma~IV.5.4]{Sil}.

\begin{lemma}\label{inverseiso}
  Let $f,g\in\Q_p[[X]]^r$ be polynomials without constant term such that $f(g(X))=X$ for $X=(X_1,\dots,X_r)$.
  Assume that $J_g(X)\in M_r(\Z_p[[X]])$ and $J_g(0)=1$. Then for all $s\in\N$ and for all $i, n_1,\dots,n_s\in\{1,\dots,r\}$ we have
  \[
    \frac{\partial^s f_i}{\partial X_{n_1}\cdots\partial X_{n_s}}(0)\in\Z_p.
  \]
\end{lemma}

\begin{proof}
  In a first step we prove the following\\ \\
  \bf Claim: For all $s \in \N$ and all $n_1,\dots,n_s\in\{1,\dots,r\}$ the expression
  \begin{equation}
    \label{claim1}
    \sum_{m_1=1}^r\cdots\sum_{m_s=1}^r\frac{\partial^s f_i}{\partial X_{m_1}\cdots
      \partial X_{m_s}}(g(X))\frac{\partial g_{m_1}}{\partial X_{n_1}}\cdots \frac{\partial g_{m_s}}{\partial X_{n_s}}
  \end{equation}
  is a polynomial in
  $\frac{\partial^t f_i}{\partial X_{k_1}\cdots\partial X_{k_t}}(g(X))$ with $1\leq t\leq s-1$, $k_1,\dots,k_t\in\{1,\dots,r\}$
  and coefficients in $\Z_p[[X]]$.

  Indeed, the chain rule for $\frac{\partial}{\partial X_{n_{1}}}$ applied to $f_i(g(X))=X_i$ yields
  \begin{equation}\label{JfgI}
    \sum_{m_1=1}^r\frac{\partial f_i}{\partial X_{m_1}}(g(X))\frac{\partial g_{m_1}}{\partial X_{n_1}}=\delta_{i,n_1}
  \end{equation}
  and thus establishes the claim for $s=1$.
  
  For the inductive step we apply $\frac{\partial}{\partial X_{n_{s+1}}}$ to the expression in (\ref{claim1}) and again by the chain rule
  we obtain
  \[
  \begin{split}
    &\sum_{m_1=1}^r\cdots\sum_{m_s=1}^r\sum_{m_{s+1}=1}^r\frac{\partial^{s+1} f_i}{\partial X_{m_1}\cdots
      \partial X_{m_s}\partial X_{m_{s+1}}}(g(X))\frac{\partial g_{m_1}}{\partial X_{n_1}}\cdots \frac{\partial g_{m_{s}}}{\partial X_{n_{s}}}
    \frac{\partial g_{m_{s+1}}}{\partial X_{n_{s+1}}} \\
  &\quad =  \frac{\partial}{\partial X_{n_{s+1}}}\left(\sum_{m_1=1}^r\cdots\sum_{m_s=1}^r\frac{\partial^s f_i}{\partial X_{m_1}\cdots\partial X_{m_s}}(g(X))\frac{\partial g_{m_1}}{\partial X_{n_1}}\cdots \frac{\partial g_{m_s}}{\partial X_{n_s}}\right)\\
  &\qquad - \sum_{m_1=1}^r\cdots\sum_{m_s=1}^r\frac{\partial^s f_i}{\partial X_{m_1}\cdots\partial X_{m_s}}(g(X))\frac{\partial}{\partial X_{n_{s+1}}}\left(\frac{\partial g_{m_1}}{\partial X_{n_1}}\cdots \frac{\partial g_{m_s}}{\partial X_{n_s}}\right).
\end{split}
\]
Using the inductive hypothesis for the first term on the right hand side and the assumption $J_g(X) \in M_r(\Zp[[X]])$ for the second one
proves the above claim.

In order to prove the assertion of the lemma we again proceed by induction on $s$. For $s=1$ we specialize (\ref{JfgI}) at $X=0$
and obtain from $g(0) = 0$
\[
  \sum_{m_1=1}^r\frac{\partial f_i}{\partial X_{m_1}}(0)\frac{\partial g_{m_1}}{\partial X_{n_1}}(0)=\delta_{i,n_1}.
\]
Since $J_g(0) = 1$ this implies $\frac{\partial f_i}{\partial X_{n_1}}(0) = \delta_{i,n_1} \in \Zp$.

For the inductive step we specialize (\ref{claim1}) at $X=0$ and since $J_g(0)=1$ we simply obtain
\[
  \frac{\partial^s f_i}{\partial X_{n_1}\cdots\partial X_{n_s}}(0).
\]
By the above claim and the inductive hypothesis this is an element in $\Zp$.
\end{proof}

\begin{lemma}\label{expconvergence}
The isomorphism $\exp_\calF$  converges for all $x=(x_1,\dots,x_r) \in L^{(r)}$ satisfying $\min\{v_L(x_1),\dots,v_L(x_r)\}>v_L(p)/(p-1)$.
\end{lemma}

\begin{proof}
By Lemma \ref{logar} and Lemma \ref{inverseiso} we have
\[
  \frac{\partial^s \exp_{\calF,i}}{\partial X_{n_1}\cdots\partial X_{n_s}}(0)\in\Z_p
\]
for any $s\in\N$ and $i, n_1,\dots,n_s\in\{1,\dots,r\}$. It follows that each component $\exp_{\calF,i}$ of $\exp_\calF$ is of the form
\[
  \sum_{m_1=0}^\infty\cdots\sum_{m_r=0}^\infty \frac{a_{m_1,\dots,m_r}}{m_1!\cdots m_r!}X_1^{m_1}\cdots X_r^{m_r}
\]
for some $a_{m_1,\dots,m_r}\in\Z_p$. As in the proof of \cite[Lemma IV.6.3 (b)]{Sil} we can show that
\[
  v_L\left(\frac{a_{m_1,\dots,m_r}}{m_1!\cdots m_r!}x_1^{m_1}\cdots x_r^{m_r}\right) \ge \sum_{i=0, m_i\ne 0}^r 
  \left( v_L(x_i) + (m_i-1) \left(v_L(x_i) - \frac{v_L(p)}{p-1}\right) \right)
\]
which under our assumption tends to infinity as the total degree tends to infinity.
\end{proof}

We summarize our discussion in the next proposition.

\begin{prop}\label{log exp iso}
  Let $L$ be a finite extension of $\Qp$ with normalized valuation $v_L$. Let $n > \frac{v_L(p)}{p-1}$ be an integer. Then
  the formal logarithm induces an isomorphism
  \[
    \log_\calF \colon \calF((\frp_L^n)^{(r)}) \lra \Gm^r((\frp_L^n)^{(r)})
  \]
  with inverse induced by $\exp_\calF$.
\end{prop}

\begin{proof}
  Given the results of this section the proposition follows as in the proof of \cite[Thm.IV.6.4]{Sil}.
\end{proof}

\end{section}

\begin{section}{\texorpdfstring{Computation of the term $\Ucris$}{Computation of the term Ucris}}

  \begin{subsection}{Some preliminary results}

    We will apply the notation introduced and explained in \cite[Sec.~1.1]{BenBer}. In particular, $\Bcris, \Bst$ and $\BdR$ denote the $p$-adic period rings constructed by Fontaine. If $V$ is a $p$-adic representation of $G_K$, we put
\[\DdR^K(V) := \left( \BdR \tensor_\Qp V \right)^{G_K}, \quad \Dcris^K(V) := \left( \Bcris \tensor_\Qp V \right)^{G_K}.\]
The $K$-vector space $\DdR^K(V)$ is finite dimensional and filtered. The tangent space of $V$ over $K$ is defined by
\[t_V(K) := \DdR^K(V) / \mathrm{Fil}^0 \DdR^K(V).\]
Finally, we write $\exp_V \colon t_V(K) \lra H^1(K, V)$ for the exponential map of Bloch and Kato. Note here that  $H^1(K, V)$ is computed
using continuous cochain cohomology (see Remark \ref{cohomology iso}).

For any $\Qp$-vector space $W$ we write $W^* = \Hom_\Qp(W, \Qp)$ for its $\Qp$-linear dual. For convenience we usually write $t^*_V(K)$ instead of $t_V(K)^*$. 

We fix a matrix $\Tur\in \Gl_r(\overline\Zpnr)$ so that $\varphi(\Tur)(\Tur)^{-1}=u^{-1}$, which exists by \cite[Lemma~1.9]{Cobbe18}.

\begin{lemma}\label{def v star}
  Let $v_1^*,\dots,v_r^*$ denote the elements of the canonical $\Qp$-basis of $V^*(1)$. Then $e_i^*=\sum_{n=1}^r (\Tur)^{-1}_{i,n}\otimes v_n^*$,
  $i = 1, \ldots, r$, constitute a basis of $\Dcris^N(V^*(1))$ as an $N_1$-vector space and of $\DdR^N(V^*(1))$ as an $N$-vector space.
  In addition, each element $e_i^*$ is fixed   by the action of the Galois group $G_\Qp$.
\end{lemma}

\begin{proof}
  The following proof is the $r$-dimensional generalization of the first part of the proof of \cite[Lemma 5.2.1]{BC2}.
  
By definition we have $\Tur=u\varphi(\Tur)$ and by induction we deduce  
$F_N(\Tur)=\varphi^{d_N}(\Tur)=u^{-d_N}\Tur$, and hence $u^{d_N}F_N(\Tur)=\Tur$. 

First of all recall that the completion $\overline\Zpnr$ of $\Zpnr$ is contained both in $\Bcris$ and $\BdR$.
{We now prove that the elements $e_i^*$ are fixed by the absolute Galois group $G_\Qp$, which will show that the $e_i^*$ are contained in both  $\Dcris^N(V^*(1))$ and $\DdR^N(V^*(1))$.}
We note that the inertia group $I_\Qp$ acts trivially on $V^*(1)$ and hence it remains to prove that $e_i^*$ is fixed by $\varphi$. We first need to calculate $\varphi(v_i^*)$. Here we use the definitions and the fact that the elements $v_i$ constitute the canonical basis of $\Q_p^r(1)(\rhonr_\Qp)$:
\[\varphi(v_i^*)(v_j)=v_i^*(u^{-1}v_j)=v_i^*\left(\sum_{k=1}^r (u^{-1})_{k,j}v_k\right)=(u^{-1})_{i,j}.\]
Hence
\[\varphi(v_i^*)=\sum_{j=1}^r (u^{-1})_{i,j}v_j^*,\]
and we conclude that
\[\begin{split}\varphi(e_i^*)&=\left(\sum_{n=1}^r (\Tur)^{-1}_{i,n}\otimes v_n^*\right)^{\varphi}=\sum_{n=1}^r \sum_{h=1}^r (\Tur)^{-1}_{i,h}u_{h,n}\otimes \sum_{k=1}^r\left(u^{-1}\right)_{n,k}v_{k}^*\\
&=\sum_{h=1}^r\sum_{k=1}^r (\Tur)^{-1}_{i,h} \delta_{h,k}\otimes v_k^*=\sum_{k=1}^r (\Tur)^{-1}_{i,k} \otimes v_k^*=e_i^*.\end{split}\]

Since $\Tur \in \Gl_r(\overline\Zpnr) \sseq \Gl_r(\Bcris)$ the elements $e_1^*,\dots,e_r^*$ are a $\Bcris$-basis of $\Bcris \tensor_\Qp V^*(1)$.
As $N_1$ is a subfield of $\Bcris$, we see that  $e_1^*,\dots,e_r^*$ are linearly independent over $N_1$.
This concludes the proof that the elements $e_i^*$ constitute a basis of $\Dcris^N(V^*(1))$ since $\dim_{N_1}\Dcris^N(V^*(1))\leq \dim_\Qp(V^*(1))=r$.

The proof that they are also a basis of the $N$-vector space  $\DdR^N(V^*(1))$ is similar.
\end{proof}

\begin{lemma}\label{defei}
  Let $v_1,\dots,v_r$ be the elements of the canonical $\Qp$-basis of $V$. The elements $e_i=\sum_{n=1}^r t^{-1} \Tur_{n,i}\otimes v_n$,
  $i = 1, \ldots, r$ constitute a basis of $\Dcris^N(V)$ as an $N_1$-vector space and of $\DdR^N(V)$ as an $N$-vector space. In addition, each element $e_i$ is fixed  by the action of the Galois group $G_\Qp$.
\end{lemma}

\begin{proof}
For $\sigma \in I_\Qp$ we compute
\[
\sigma(e_i) = \sum_{n=1}^r \sigma(t^{-1} T^{nr}_{n,i}) \tensor \sigma(v_n) = 
\sum_{n=1}^r \chi_{cyc}(\sigma^{-1})t^{-1} T^{nr}_{n,i} \tensor \chi_{cyc}(\sigma)v_n = e_i.
\]
Hence the elements $e_i$ are fixed by the inertia group and a  similar computation as in the proof of Lemma \ref{def v star} 
shows that $\varphi(e_i) = e_i$. The proof follows as above.
\end{proof}

\begin{lemma}\label{basisV1}
Let $\tilde v_1,\dots,\tilde v_r$ be the elements of the canonical $\Qp$-basis of $V(-1)$. The elements $\tilde e_i=\sum_{n=1}^r  \Tur_{n,i}\otimes \tilde v_n$ are a basis of $\Dcris^N(V(-1))$ as an $N_1$-vector space and of $\DdR^N(V(-1))$ as an $N$-vector space. In addition, each element $\tilde e_i$ is fixed  by the action of the Galois group $G_\Qp$.
\end{lemma}

\begin{proof}
Similar as above.
\end{proof}
\end{subsection}

\begin{subsection}{\texorpdfstring{Computation of $\Ucris$}{Computation of Ucris}}
We recall that $V = \Q_p^r(1)(\rhonr)$ and $V^*(1) = \Q_p^r((\rhonr)^{-1})$. The following lemma (and its proof) is the analogue of \cite[Lemma 5.1.2]{BC2}.
\begin{lemma}\label{degenerate}
We have:
\begin{enumerate}[(a)]
\item $t_{V^*(1)}(N)=0$.
\item $H_f^1(N,V^*(1))=0$.
\item $H_f^1(N,V)= H^1_e(N,V) = H^1(N,V)$.
\end{enumerate}
\end{lemma}

\begin{proof}
Proofs are as for \cite[Lemma 5.1.2]{BC2}. Note that in the present manuscript we always assume that $\rhonr |_N \ne 1$. For the proof of 
part (c) we also need that by Lemma \ref{lemmaforUcris 2} below the endomorphism $1 - \phi$ of $\Dcris(V)$ is an isomorphism. 
\end{proof}

By the above lemma, \cite[Coroll. 3.16]{Cobbe18} and \cite[(30)]{BC2} the $7$-term exact sequence \cite[(5)]{BC2} degenerates into the two exact sequences 
\[
0 \lra \Dcris^N(V) \xrightarrow{1- \phi} \Dcris^N(V) \oplus t_V(N) \lra H^1(N, V) \lra 0
\]
and 
\[
0 \lra \Dcris^N(V^*(1))^* \xrightarrow{1- \phi^*} \Dcris^N(V^*(1))^* \lra 0.
\] 
As in \cite{BC2}, see in particular equation (32) of loc.cit., we obtain
\begin{equation}\label{yar 4}
\Ucris = \hat\partial_{\ZpG, \QpG}^1([\Dcris^N(V), 1 - \phi]) - \hat\partial_{\ZpG, \QpG}^1([\Dcris^N(V^*(1))^*, 1 - \phi^*]).
\end{equation}
Before computing $\Ucris$ we need an easy lemma from linear algebra.

\begin{lemma}\label{detUcris}
Let
\[
M=\begin{pmatrix}1&0&0&\cdots&0&B_1\\
A&1&0&\cdots&0&B_2\\
0&A&1&\cdots&0&B_3\\
\vdots&\vdots&\vdots&\ddots&\vdots&\vdots\\
0&0&0&\cdots&1&B_{n-1}\\
0&0&0&\cdots&A&1+B_n\end{pmatrix}
\]
be a block matrix, with $n^2$ square blocks of the same size. Then
\[\det(M)=\det\left(1+\sum_{i=0}^{n-1}(-A)^{i}B_{n-i}\right).\]
\end{lemma}

\begin{proof}
By Gaussian elimination we obtain the matrix
\[
\begin{pmatrix}1&0&0&\cdots&0&B_1\\
0&1&0&\cdots&0&B_2-AB_1\\
0&0&1&\cdots&0&B_3-AB_2+A^2B_1\\
\vdots&\vdots&\vdots&\ddots&\vdots&\vdots\\
0&0&0&\cdots&1&\sum_{i=0}^{n-2}(-A)^i B_{n-1-i}\\
0&0&0&\cdots&0&1+\sum_{i=0}^{n-1}(-A)^i B_{n-i}\end{pmatrix}.
\]
\end{proof}

We introduce the following notation. If $x \in Z(\QpG)$ we let ${}^*x \in Z(\QpG)^\times$ denote the invertible element which
on the Wedderburn decomposition $Z(\QpG) = \bigoplus_{i=1}^r F_i$ for suitable finite extensions $F_i/\Qp$ is given by $({}^*x_i)$ with
${}^*x_i = 1$ if $x_i = 0$ and ${}^*x_i = x_i$ otherwise.

We now  generalize \cite[Lemma 5.2.1 and Lemma 5.2.2]{BC2} and in this way explicitly compute the element $\Ucris$.
Recall that $F = F_K= \varphi^{d_K}$ is the Frobenius element of $K$. We write $I = I_{N/K}$ for the inertia subgroup of the
Galois extension $N/K$.

\begin{lemma}\label{lemmaforUcris 1}
The endomorphism $1-\phi^*$ of $\Dcris^N(V^*(1))^*$ is an isomorphism. Furthermore we have
\[\hat\partial_{\ZpG, \QpG}^1([\Dcris^N(V^*(1))^*\!, 1 - \phi^*]) = 
\hat\partial_{\ZpG, \QpG}^1({}^*(\det(1-u^{d_K}F^{-1})e_I))\]
in $K_0(\ZpG,\!\QpG)$.
\end{lemma}

\begin{proof} We have to compute $\phi(e_i^*)$.
  Using the $\varphi$-semilinearity of $\phi$ we compute
\[\begin{split}\phi(e_i^*)&=\sum_{n=1}^r\varphi((\Tur)^{-1})_{i,n}\otimes v_n^*=\sum_{n=1}^r ((\Tur)^{-1} u)_{i,n}\otimes v_n^*\\
&=\sum_{n=1}^r ((\Tur)^{-1} u\Tur(\Tur)^{-1})_{i,n}\otimes v_n^*=\sum_{n=1}^r\sum_{\ell=1}^r ((\Tur)^{-1} u \Tur)_{i,\ell}(\Tur)^{-1}_{\ell,n}\otimes v_n^*\\
&=\sum_{\ell=1}^r ((\Tur)^{-1} u\Tur)_{i,\ell}e_\ell^*.\end{split}\]

We fix a normal basis element $\theta$ of $N_1/\Q_p$. Then $w_{i,j}:=\varphi^{-j}(\theta) e_i^*$ for $i=1,\dots,r$ and $j=0,\dots,d_K-1$ is a $\Q_p[G/I]$-basis of $\Dcris^N(V^*(1))$. Let $\psi_{i,j}\in \Dcris^N(V^*(1))^*$ for $i=1,\dots,r$ and $j=0,\dots,d_K-1$ be the dual $\Q_p[G/I]$-basis.

For $0<i\leq r$, $0\leq j<d_K-1$ and $0\leq n\leq d-1$ we have 
\[\begin{split}&\phi^*(\psi_{i,j})(F^n w_{h,k})=\psi_{i,j}(\phi(F^n \varphi^{-k}(\theta) e_h^*))=\psi_{i,j}\!\left(\!F^n \varphi^{-k+1}(\theta) \sum_{\ell=1}^r ((\Tur)^{-1} u\Tur)_{h,\ell}e_\ell^*\right)\\
&\qquad=\sum_{\ell=1}^r ((\Tur)^{-1} u \Tur)_{h,\ell}\psi_{i,j}(F^n \varphi^{-k+1}(\theta) e_\ell^*)=\sum_{\ell=1}^r ((\Tur)^{-1} u\Tur)_{h,\ell}\psi_{i,j}(F^n w_{\ell,k-1})\end{split}\]
If $n=0$, $k=j+1$ and any $0<h\leq r$ this is equal to $((\Tur)^{-1} u\Tur)_{h,i}$; it is $0$ otherwise.
Hence
\[\phi^*(\psi_{i,j})=\sum_{h=1}^r ((\Tur)^{-1} u\Tur)_{h,i}\psi_{h,j+1}.\]
Analogously for $0<i\leq r$ and $j=d_K-1$ we have
\[
  \phi^*(\psi_{i,d_K-1})=\sum_{h=1}^r ((\Tur)^{-1} u \Tur)_{h,i}F^{-1}\psi_{h,0}.
\]
With respect to the $\Qp[G/I]$-basis $\psi_{i,j}$ the matrix associated to $1-\phi^*$ is given by
\[
\begin{pmatrix}1&0&0&\cdots&0&-F^{-1}(\Tur)^{-1} u\Tur\\
-(\Tur)^{-1} u\Tur&1&0&\cdots&0&0\\
0&-(\Tur)^{-1} u\Tur&1&\cdots&0&0\\
\vdots&\vdots&\vdots&\ddots&\vdots&\vdots\\
0&0&0&\cdots&1&0\\
0&0&0&\cdots&-(\Tur)^{-1} u\Tur&1\end{pmatrix}.
\]
By Lemma \ref{detUcris}, its determinant is $\det(1-F^{-1}(\Tur)^{-1} u^{d_K}\Tur)=\det(1-F^{-1}u^{d_K})$.

{To conclude that $1-\phi^*$ is an isomorphism it is now enough to notice that 
  \[
    (1-F^{-1}u^{d_K})(1+F^{-1}u^{d_K}+\dots+(F^{-1}u^{d_K})^{d_{N/K-1}})=1-U_N,
  \]
which is invertible since by Hypothesis (F) we always have $\det(U_N-1)\neq 0$.}
\end{proof}

\begin{lemma}\label{lemmaforUcris 2}
The endomorphism $1 - \phi$ of $\Dcris^N(V)$ is an isomorphism and we have
\[
\hat\partial_{\ZpG, \QpG}^1([\Dcris^N(V), 1 - \phi]) = 
\hat\partial_{\ZpG, \QpG}^1({}^*(\det(1-F(pu)^{-d_K})e_I))
\]
in $K_0(\ZpG, \QpG)$.
\end{lemma}

\begin{proof}
  The proof is analogous to the proof of the previous lemma.
  
  To show that $\det(1-F(pu)^{-d_K})$ is invertible in $\Qp[G/I]$, it is enough to notice that
    $\sum_{i=0}^\infty (F^{-1}(pu)^{d_K})^i$ converges in the matrix ring $M_r(\Zp[G/I])$ and is the inverse of
    \[
      1-F^{-1}(pu)^{d_K}=-F^{-1}(pu)^{d_K}(1-F(pu)^{-d_K}).
    \]

\end{proof}

\begin{prop}\label{propUcris}
We have
\[
  \Ucris =\hat\partial_{\ZpG, \QpG}^1({}^*(\det(1-F(pu)^{-d_K})e_I))-\hat\partial_{\ZpG, \QpG}^1({}^*(\det(1-u^{d_K}F^{-1})e_I)).
\]
\end{prop}

\begin{proof}
The proof is easily achieved by combining (\ref{yar 4}), Lemma \ref{lemmaforUcris 1} and Lemma \ref{lemmaforUcris 2}.
\end{proof}

We conclude this section proving some functorial properties for the term $\Ucris$.
To that end, we let $L$ be an intermediate field of $N/K$ and set 
$H := \Gal(N/L)$. Then restriction induces a homomorphism
\begin{equation}\label{K0 res}
\rho_H^G \colon K_0(\ZpG, \QpG) \lra K_0(\Zp[H], \Qp[H]),
\end{equation}
and if $H$ is normal in $G$, then we also have a 
 homomorphism
\begin{equation}\label{K0 quot}
q_{G/H}^G \colon K_0(\ZpG, \QpG) \lra K_0(\Zp[G/H], \Qp[G/H]).
\end{equation}
\begin{lemma}\label{functUcris}
Let $L$ be an intermediate field of $N/K$ and $H=\Gal(N/L)$. Then:
\begin{enumerate}[(a)]
\item $\rho^G_H(U_{\mathrm{cris},N/K})=U_{\mathrm{cris},N/L}$.
\item If $H$ is normal in $G$, then $q^G_{G/H}(U_{\mathrm{cris},N/K})=U_{\mathrm{cris},L/K}$.
\end{enumerate}
\end{lemma}

\begin{proof}
  Let $u_{\mathrm{cris},N/K}\in Z(\Qp[G])^\times$  be such that $\hat\partial_{\ZpG, \QpG}^1(u_{\mathrm{cris},N/K})=U_{\mathrm{cris},N/K}$.
  We will use an analogous notation for all the other Galois extensions involved in the proof.

  Then, for any irreducible character $\chi$ of $G$ we have
  \[
    (u_{\mathrm{cris},N/K})_\chi=
    \begin{cases}
      \frac{\det(1-\chi(F)(pu)^{-d_K})}{\det(1-u^{d_K}\chi(F)^{-1})}&\text{if }\chi|_{I_{N/K}}=1,\\
       1&\text{if }\chi|_{I_{N/K}} \neq 1.
     \end{cases}
   \]

   For two (virtual) characters $\chi_1$ and $\chi_2$ of a finite group $J$ we write $\langle \chi_1, \chi_2\rangle_J$
   for the standard scalar product.

\begin{enumerate}[(a)] 
\item  By \cite[Sec.~6.1]{BleyWilson} we have
  \[
    \rho^G_H(u_{\mathrm{cris},N/K})=
    \left( \prod_{\chi\in\mathrm{Irr}(G)}(u_{\mathrm{cris},N/K})_\chi^{\langle\chi,\mathrm{Ind}^G_H\psi\rangle_G}
    \right)_{\psi\in \Irr({H})}.
  \]
  Since $\langle\chi,\mathrm{Ind}^G_H\psi\rangle_G = \langle \chi|_H,\psi\rangle_H$ by Frobenius reciprocity we obtain
  \[
    (u_{\mathrm{cris},N/K})_\chi^{\langle \chi,\mathrm{Ind}_H^G\psi\rangle_G}
    =\begin{cases}1&\text{if }\chi|_{I_{N/K}}\neq 1,\\
                 u_{\mathrm{cris},N/K}& \text{if }\chi|_{I_{N/K}}=1 \text{ and } \chi|_H = \psi,\\
                 1&\text{if }\chi|_{I_{N/K}}=1 \text{ and } \chi|_H \neq \psi.
               \end{cases}
  \]
  If $\psi|_{I_{N/L}} \ne 1$, then $\chi|_{I_{N/K}} \ne 1$ whenever $\langle \chi|_H, \psi \rangle_H \ne 0$.
  Thus $\rho^G_H(u_{\mathrm{cris},N/K})_\psi = 1$ for those characters $\psi$.

  On the other hand, if $\psi|_{I_{N/L}} = 1$, then $\psi$ is a character of the cyclic group $\overline{H} := H/I_{N/L} = \langle F_L \rangle$.
  Each character $\chi \in \Irr(G)$ with $(u_{\mathrm{cris},N/K})_\chi \ne 1$ is actually a character of
  $\overline{G} := G / I_{N/K} = \langle F_K \rangle$. Note that we can naturally identify $\overline{H}$ with a subgroup of
  $\overline{G}$ and recall that $|\overline{G} / \overline{H}| = d_{L/K}$.

  We therefore obtain
  \[
    \prod_{\chi\in\mathrm{Irr}(G)}(u_{\mathrm{cris},N/K})_\chi^{\langle\chi,\mathrm{Ind}^G_H\psi\rangle_G} =
    \prod_{\chi \in \Irr(\overline{G}), \chi|_{\overline{H}} = \psi} \frac{\det(1 - \chi(F_K)(pu)^{-d_K})}{\det(1 - u^{d_K}\chi(F_K)^{-1})}.
  \]
  We consider the numerator and denominator separately and use in each case the polynomial identity
  \[
    \prod_{\chi \in \Irr(\overline{G}), \chi|_{\overline{H}} = \psi} (X - \chi(F_K)) = X^{d_{L/K}} - \psi(F_L).
  \]
  For the numerator we compute
  \begin{eqnarray*}
    &&   \prod_{\chi \in \Irr(\overline{G}), \chi|_{\overline{H}} = \psi} \det(1 - \chi(F_K)(pu)^{-d_K}) \\
    &=& \det \left( \prod_{\chi \in \Irr(\overline{G}), \chi|_{\overline{H}} = \psi} (pu)^{-d_K} \left( (pu)^{d_K} - \chi(F_K) \right) \right) \\
    &=& \det \left( (pu)^{-d_K d_{L/K}} \left( (pu)^{d_K d_{L/K}} - \chi(F_L) \right) \right) \\
    &=& \det \left( 1 - (pu)^{-d_L} \psi(F_L)\right)
  \end{eqnarray*}
  and a similar computation for the denominator shows claim (a).

\item For any character $\psi$ of $G/H$ we write $\mathrm{infl}(\psi)$ for the inflated character of $G$. By \cite[Sec.~6.3]{BleyWilson},
$q^G_{G/H}(u_{\mathrm{cris},N/K})_\psi=(u_{\mathrm{cris},N/K})_{\mathrm{infl}(\psi)}$ for any $\psi \in \Irr(G/H)$.
This is equal to $(u_{\mathrm{cris},L/K})_\psi$ because $I_{L/K} = I_{N/K}H/H$ and $\mathrm{infl}(\psi)(F_K)=\psi(F_K)$.
\end{enumerate}
\end{proof}

\end{subsection}
\end{section}

\begin{section}{Computation of epsilon constants}

As in \cite[Sec.~2.3]{IV} we define
\[
\epsilon_D(N/K,V)=(\epsilon(D_\mathrm{pst}(\mathrm{Ind}_{K/\Q_p}(V\otimes\rho_\chi^*)), \psi_\xi, \mu_\Qp))_{\chi\in\mathrm{Irr}(G)}
\in\!\!\!\!\prod_{\chi\in\mathrm{Irr}(G)}\overline{\Qp}^\times \cong Z({\Qpc}[G])^\times.
\]
For all unexplained notation we refer the reader to \cite[Sec.~2.3]{IV}. If there is no danger of confusion we sometimes drop $\psi_\xi$ and $\mu_\Qp$ 
from our notation.
Still following \cite{IV} (see the proof of Lemma 4.1 of loc.cit.), we obtain
\begin{equation}\label{Dpst iso}
  D_\mathrm{pst}(\mathrm{Ind}_{K/\Qp}(V\otimes\rho_\chi^*))\cong
  D_\mathrm{pst}(V)\otimes_{\Q_p^\mathrm{nr}} D_\mathrm{pst}(\mathrm{Ind}_{K/\Qp}(\chi^*)).
\end{equation}

For an extension $L/K$ of $p$-adic fields we write $\frD_{L/K} = \pi_L^{s_{L/K}}\OL$
for the different of $L/K$; in the case $K=\Q_p$ we use the notations $\frD_L$ for $\frD_{L/\Q_p}$ and $s_L$ for $s_{L/\Q_p}$.
If  $M/L$ is a finite abelian extension and $\eta$ an irreducible character of $\Gal(M/L)$, then we let 
$\tau_L(\eta)$ denote the abelian local Galois Gau\ss\ sum defined, e.g., in  \cite[page~1184]{PickettVinatier}).
For the definition of Galois Gau\ss\ sums for a finite Galois extension
$M/L$ we refer the reader to \cite[I.\S 5]{Froehlich83}.

\begin{prop}\label{eps coroll}
For each $\chi \in \Irr(G)$ we have the equality
\[\begin{split}
&\hat\partial_{\ZpG, \Qpc[G]}^1\left( \epsilon_D(N/K, V)\right) \\
&\qquad= 
\hat\partial_{\ZpG, \Qpc[G]}^1\left( \left( \det(u)^{-d_K(s_K\chi(1) + m_\chi)} \cdot \tau_\Qp( \Ind_{K/\Qp}(\chi) )^{-r} \right)_{\chi \in \Irr(G)} \right).\end{split}
\]
where we write $\frf(\chi) = \pi_K^{m_\chi}\OK$ for the Artin conductor of $\chi$.
\end{prop}

\begin{proof}
  In the proof we will use the list of properties in \cite[Sec.~2.3]{BenBer}. The field $K$ of loc.cit.
  corresponds to $\Qp$ in our situation. Hence,
  if $\psi$ denotes the standard additive character, we have $n(\psi) = 0$ for its conductor.   
  
Since $V$ is cristalline the $N_1$-basis $\{e_i\}$ constructed in Lemma  \ref{defei} is also a $\Qpnr$-basis of $D_\mathrm{pst}(V)$.
A straightforward computation (see the proof of Lemma \ref{lemmaforUcris 1} for a similar computation) shows that 
\begin{equation}\label{phi on ei}
\phi(e_i) = \sum_{h=1}^r p^{-1} ((\Tnr)^{-1} u^{-1} \Tnr)_{h,i}e_h.
\end{equation}
As in the proof of Lemma \ref{defei}, any element $\sigma$ of the absolute inertia group acts trivially on the
basis elements $e_i$, whence $D_\mathrm{pst}(V)$ is unramified. 

Applying (\ref{Dpst iso}) and \cite[2.3 (6)]{BenBer} we obtain
\[\begin{split}
& \epsilon(D_\mathrm{pst}(\mathrm{Ind}_{K/\Q_p}(V\otimes\rho_\chi^*)),\psi_\xi) \\
&\qquad= \epsilon(D_\mathrm{pst}(\mathrm{Ind}_{K/\Qp}(\chi^*)),\psi_\xi)^{r}
\cdot \det(D_\mathrm{pst}(V))(p)^{m(D_\mathrm{pst}(\mathrm{Ind}_{K/\Qp}(\chi^*)))},
\end{split}\]
where $m(D_\mathrm{pst}(\mathrm{Ind}_{K/\Qp}(\chi^{{*}})))$ is the exponent of the Artin conductor.

We consider the first factor and note that
\[
D_\mathrm{pst}(\mathrm{Ind}_{K/\Qp}(\chi^*)) \cong \Qpnr \tensor_\Qp  \mathrm{Ind}_{K/\Qp}(\chi^*),
\]
so that we deduce from \cite[Prop.~6.1.3]{BC2} that
\[
\epsilon(D_\mathrm{pst}(\mathrm{Ind}_{K/\Qp}(\chi^*)),\psi_\xi) =  \tau_\Qp( \Ind_{K/\Qp}(\chi^*) ).
\]
As a consequence of \cite[Prop.~II.4.1 (ii)]{Martinet77} we get
\[
  \tau_\Qp( \Ind_{K/\Qp}(\chi^*) ) = \tau_\Qp( \Ind_{K/\Qp}(\chi) )^{-1} \cdot p^{m(D_\mathrm{pst}(\mathrm{Ind}_{K/\Qp}(\chi^*)))}
  \cdot (\det((\Ind_{K/\Qp}\chi)(-1))).
\]

For the second factor, we first note that
\[
  \det(D_\mathrm{pst}(V))(p) = \det(\varphi^{-1}, D_\mathrm{pst}(V)).
\]
By  \cite[page 625]{BenBer} the action of the Weil group $W_{\Qp}$ on $D_\mathrm{pst}(V)$ is defined so that the action of
the geometric Frobenius $\varphi^{-1}$ coincides with the usual action of $\varphi^{-1} \phi$ on $D_\mathrm{pst}(V)$.
We recall from (\ref{phi on ei}) that with respect to the
basis $\{e_i\}$ of $D_\mathrm{pst}(V)$ the element $\varphi^{-1}\phi$
acts as $\varphi^{-1}( p^{-1}(\Tnr)^{-1} u^{-1}\Tnr)$ on $D_\mathrm{pst}(V)$,
so that we derive $\det(\varphi^{-1},D_\mathrm{pst}(V)) = p^{-r}\det(u^{-1})$.

Finally, by \cite[VII.11.7]{Neukirch92}, we get
\[
  \frf( \Qpnr \tensor_\Qp \mathrm{Ind}_{K/\Qp}(\chi^*) ) =
  \frf( \mathrm{Ind}_{K/\Qp}(\chi^*) ) = \frd_K^{\chi(1)} N_{K/\Qp}(\frf(\chi^*)) = p^{d_K ( s_K\chi(1) + m_\chi) },
\]
so that 
\[
  m(D_\mathrm{pst}(\mathrm{Ind}_{K/\Qp}(\chi^*)))=
  m(\Qpnr \tensor_\Qp  \mathrm{Ind}_{K/\Qp}(\chi^*)) = d_K ( s_K\chi(1) + m_\chi).
\]

We conclude that
\[
\epsilon_D(N/K, V)_\chi =  \left(  \det(\Ind_{K/\Qp}(\chi))(-1)  \right)^{r} \cdot 
\det(u)^{-d_K(s_K\chi(1) + m_\chi)} \cdot \tau_\Qp( \Ind_{K/\Qp}(\chi) )^{-r}.
\]

The proposition is now immediate from \cite[Lemma~6.2.2]{BC2}, which shows that
\[ 
\hat\partial_{\ZpG, \Qp[G]}^1   \left( \left( \det(\Ind_{K/\Qp}\chi)(-1) \right)_{\chi\in \Irr(G)}\right) = 0.
\]
\end{proof}

Concerning functoriality with respect to change of fields we have the following lemma.

\begin{lemma}\label{functepsilon}
Let $L$ be an intermediate field of $N/K$ and $H=\Gal(N/L)$.
\begin{enumerate}[(a)]
\item $\rho^G_H(\hat\partial^1_{\ZpG, \QpcG}(\epsilon_D(N/K, V)))=\hat\partial^1_{\Zp[H], \Qpc[H]}(\epsilon_D(N/L, V))$.
\item If $H$ is normal in $G$, then 
\[q^G_{G/H}(\hat\partial^1_{\ZpG, \QpcG}(\epsilon_D(N/K, V)))=\hat\partial^1_{\Zp[G/H], \Qpc[G/H]}(\epsilon_D(L/K, V)).\]
\end{enumerate}
\end{lemma}

\begin{proof}
By Proposition \ref{eps coroll} we have 

\begin{eqnarray*}
  && \rho^G_H(\hat\partial^1_{\ZpG, \QpcG}(\epsilon_D(N/K, V))) \\
  &=& \rho^G_H\left(\hat\partial^1_{\ZpG, \QpcG} \left( \left( \det(u)^{-d_K(s_K\chi(1) + m_\chi)} 
      \cdot \tau_\Qp( \Ind_{K/\Qp}(\chi) )^{-r} \right)_{\chi \in \Irr(G)} \right) \right).
\end{eqnarray*}

By \cite[Lemma~2.3]{Breuning04} we have
\[
\rho^G_H(\hat\partial^1_{\ZpG, \QpcG} \left( \tau_\Qp( \Ind_{K/\Qp}(\chi) )^{-r} \right)_{\chi \in \Irr(G)} )\! =\!
\hat\partial^1_{\Zp[H], \Qpc[H]} \left( \tau_\Qp( \Ind_{L/\Qp}(\psi) )^{-r} \right)_{\psi \in \Irr(H)},
\]

whereas  \cite[Sec.~6.1]{BleyWilson} implies
 \[
\rho^G_H\left(\hat\partial^1_{\ZpG, \QpcG} \left( \det(u)^{-d_K(s_K\chi(1) + m_\chi)} \right) \right)
= \hat\partial^1_{\Zp[H], \Qpc[H]} \left((\alpha_\psi) _{\psi \in \Irr(H)}\right)
\]
with 
\[
\alpha_\psi = \prod_{\chi \in \Irr(G)}  \det(u)^{-d_K(s_K\chi(1) + m_\chi) \langle \chi, \Ind_H^G\psi \rangle_G}.
\]
From \cite[Thm.~VII.11.7]{Neukirch92} and the obvious relation
\[\Ind_H^G\psi = \sum_{\chi \in \Irr(G)} \langle \chi, \Ind_H^G\psi\rangle_G \chi\]
we derive
\[
 \prod_{\chi \in \Irr(G)} \frf(N/K, \chi)^{ \langle \chi, \Ind_H^G\psi\rangle_G} 
= \frf(N/K, \Ind_H^G\psi) = \frd_{L/K}^{\psi(1)} N_{L/K}(\frf(N/L, \psi)),
\]
{where $\frd_{L/K}$ denotes the discriminant of $L/K$.} This implies   
\[
\sum_{\chi \in \Irr(G)} m_\chi \langle \chi, \Ind_H^G\psi\rangle_G = d_{L/K}s_{L/K}\psi(1) + d_{L/K}m_\psi.
\]
Furthermore, we note
\[
\sum_{\chi \in \Irr(G)} \langle \chi, \Ind_H^G\psi\rangle_G \cdot \chi (1) = (\Ind_H^G\psi)(1) = [G:H]\psi(1) = e_{L/K}d_{L/K}{\psi(1)}.
\]
Hence we deduce from $d_L = d_Kd_{L/K}$
\[
  \begin{split}
    &\sum_{\chi\in\mathrm{Irr}(G)}d_K(s_K\chi(1) + m_\chi)\langle\chi,\mathrm{ind}^G_H\psi\rangle_G\\
&\qquad=d_K( s_Ke_{L/K}d_{L/K}\psi(1)+d_{L/K}s_{L/K}\psi(1)+d_{L/K}m_\psi)\\
&\qquad=d_L s_K e_{L/K}\psi(1)+d_Ls_{L/K}\psi(1)+d_{L}m_\psi\\
&\qquad=d_L \psi(1) (s_K e_{L/K}+s_{L/K})+d_{L}m_\psi\\
&\qquad=d_L \psi(1)s_L+d_{L}m_\psi.
\end{split}
\]
where the last equality follows from the multiplicativity of differents, see \cite[Thm.~III.2.2 (i)]{Neukirch92}.
The first functoriality property  is now obvious.

\smallskip

The second functoriality property follows easily from  \cite[Lemma~2.3]{Breuning04} and \cite[Lemma~VII.11.7~(ii)]{Neukirch92}.
\end{proof}

\end{section}

\begin{section}{Computation of the cohomological term}
  \begin{subsection}{Identifying cohomology}\label{Preliminary results}
    In the following we will take the opportunity to clarify some of the constructions of \cite[Sec.~7.1, page 356]{BC2}.
    This is necessary since in the definition of $C_{N/K}$ we use the identification of $H^1(N, T)$ with $\calF(\frp_N^{(r)})$ coming from
    continuous cochain cohomology combined
    Kummer theory whereas in the computations in \cite[Sec.~7.1]{BC2} we use the identification coming from
    \cite[Thm.~4.3.1]{BC2} combined with \cite[Lemma~4.1.1]{BC2}. In this manuscript we work  in the
    $r$-dimensional setting based on the results of
    \cite[Sec.~3]{Cobbe18}, where the special case $r=1$ covers the situation of \cite{BC2}.

    Let
    \[
      C_{N, \rhonr}^\bullet := [ \calI_{N/K}(\chinr) \lra \calI_{N/K}(\chinr) \lra \Q_p^r/(F_N-1)\cdot\Q_p^r(\rhonr) ],
    \]
     with non-trivial modules in degree $0$, $1$ and $2$, be the complex of \cite[Thm.~3.12]{Cobbe18} and let
    \[
      C_{N,T}^\bullet := [ \calI_{N/K}(\chinr) \lra \calI_{N/K}(\chinr) ],
    \]
    with non-trivial modules in degree $1$ and $2$, be the complex of \cite[Thm.~3.15]{Cobbe18}.
		We also deduce from \cite[Thm.~3.3]{Cobbe18} combined
    with \cite[Lemma~2.1]{Cobbe18} the short exact sequence
    \begin{equation}\label{fFN}
      0 \lra \calF(\frp_N^{(r)}) \xrightarrow{f_{\calF,N}} \calI_{N/K}(\chinr) \lra \calI_{N/K}(\chinr) \lra \calZ / (F_N - 1)\calZ \lra 0.
    \end{equation}
    
    In the sequel we will use dotted arrows for morphisms in the derived category and solid arrows for those which are
    actual morphisms of complexes.

    In the proof of \cite[Thm.3.12]{Cobbe18} we construct an isomorphism
    \[
      \xymatrix{\tau \colon C_{N, \rhonr}^\bullet \ar@{..>}[r] & R\Gamma(N,\calF)}
    \]
    in the derived category which induces the identity on $H^0$. In a second step, see \cite[Cor.~3.13]{Cobbe18},
    we produce quasi-isomorphisms
    \[
      \eta \colon P^\bullet \stackrel\sim\lra C_{N, \rhonr}^\bullet
\text{ and }
      \tilde\eta \colon \tilde P^\bullet \stackrel\sim\lra C_{N,T}^\bullet[1]
    \]
    where
\begin{eqnarray*}
  P^\bullet &:=& [P^{-1} \lra P^0 \lra P^1\lra \Q_p^r/(F_N-1)\cdot\Q_p^r(\rhonr)], \\
  \tilde P^\bullet &:=& [P^{-1} \lra P^0 \lra P^1].
\end{eqnarray*}
Here the $\Z_p[G]$-modules $P^{-1}, P^{0}, P^1$ are finitely generated and projective and the uniquely divisible module
$\Q_p^r/(F_N-1)\cdot\Q_p^r(\rhonr)$ is $G$-cohomologically trivial. Composing $\eta$ and $\tau$ we obtain an isomorphism
    $\xymatrix{P^\bullet \ar@{..>}[r] & R\Gamma(N, \calF)}$
    in the derived category. 
  Passing to the projective limit in  \cite[Lemma 3.14]{Cobbe18}  we obtain another quasi-isomorphism
  $\varphi \colon \tilde P^\bullet \stackrel\sim\lra R\Gamma(N, T)[1]$
  and thus obtain the following commutative diagram in the derived category
		\begin{equation}\label{clarification1}
   \xymatrix{
      C_{N, \rhonr}^\bullet \ar@{..>}@/_2pc/[dd]_{\tau}&& \ar[ll]_{\iota} C_{N, T}^\bullet[1]  \ar@{..>}@/^2pc/[dd]^-{\xi} \\
      P^\bullet \ar[u]^\sim_\eta \ar@{..>}[d]^{\tau\circ\eta}_\sim && \ar[ll]_{\iota} \tilde P^\bullet \ar[u]^\sim_{\tilde\eta} \ar[d]_\sim^\varphi \\
      R\Gamma(N, \calF)  && \ar@{..>}[ll]_-{\tau\circ\eta\circ\varphi^{-1}} R\Gamma(N, T)[1]
    }
  \end{equation}
  On $H^0$ we therefore obtain the commutative diagram
  \begin{equation}\label{clarification2}
    \xymatrix{
      \calF(\frp_N^{(r)}) \ar[rr]^{f_{\calF,N}} \ar[d]^{\id} && H^0(C_{N, \rhonr}^\bullet)  \ar[d]^{H^0(\tau)} \ar[rr]^{\id} && H^1(C_{N, T}^\bullet) \ar[d]^{H^0(\xi)} \\
      \calF(\frp_N^{(r)}) \ar[rr]^{\id}  && H^0(N,\calF) \ar[rr]^{H^0(\varphi\circ\eta^{-1}\circ\tau^{-1})}  && H^1(N, T)
    }
  \end{equation}
  By the proof of \cite[Thm.~4.3.1]{BC2} (which is used also for \cite[Thm.~3.15]{Cobbe18}) we know that the composite
  \[
    H^0(N, \calF) \xrightarrow{H^0(\varphi\circ\eta^{-1})} H^1(N, T) \lra H^1(N, \calF[p^n])
  \]
  is the Kummer map $\partial_{Ku,n}$ resulting from the distinguished triangle
  \[
    R\Gamma(N, \calF) \stackrel{p^n}\lra R\Gamma(N, \calF) \lra R\Gamma(N, \calF[p^n])[1] \lra .
  \]
  By the universal property of projective limits we obtain
  \[
    H^0(\varphi\circ\eta^{-1}) = \partial_{Ku}, \text{ respectively } H^0(\xi) = \partial_{Ku}^{}.
  \]
\end{subsection}

\begin{subsection}{Definition of the twist invariant}\label{Utwist def}

  In this subsection we define an invariant $\Utwist$ in the relative algebraic $K$-group $K_0(\ZpG, \overline\Qpnr[G])$.
  We recall that $\Tnr \in \Gl_r(\overline{\Zpnr})$ satisfies the matrix equality
  \begin{equation}\label{Tnr property}
    \varphi(\Tnr) = u^{-1} \Tnr.    
  \end{equation}
  This equality determines $\Tnr$ up to right multiplication by a matrix $S \in \Gl_r(\Zp)$, explicitly, if $\tilde{T}^\mathrm{nr}$ is a
  second matrix satisfying (\ref{Tnr property}), then $\Tnr = \tilde{T}^\mathrm{nr}S$.
  It is thus immediate that the element
  \begin{equation}\label{Utwist}
    \Utwist := \hat\partial_{\ZpG, \overline\Qpnr[G]}^1(\det(\Tnr))
  \end{equation}
  does not depend on the specific choice of $\Tnr$ satisfying (\ref{Tnr property}).
    \begin{remark}\label{Utwist remark}
      The element $\Utwist$ clearly becomes trivial under the canonical map
    $K_0(\ZpG,  \overline\Qpnr[G]) \lra  K_0(\overline{\Zpnr}[G],  \overline\Qpnr[G])$.
    \end{remark}
  
\end{subsection}

\begin{subsection}{\texorpdfstring{The cohomological term $C_{N/K}$}{The cohomological term}}
  In this subsection we clarify and correct the computation of the cohomological term $C_{N/K}$ of \cite[Sec.~7.1]{BC2}.
    In particular, we produce a detailed proof of \cite[Lemma 7.1.2]{BC2}, which in loc.cit. was quoted from \cite[Lemma 6.1]{IV}.
    It is this part of the computation where the new term $\Utwist$ emerges.

  We recall that we throughout assume Hypothesis (F), in particular, $\rhonr |_{G_N} \ne 1$.
  Then, by \cite[(15), (16)]{BC2} the cohomological term $C_{N/K}$ is
  defined by 
\begin{equation}\label{CNK}
C_{N/K}=-\chi_{\ZpG,\BdRG}(M^\bullet,\exp_V\circ\comp_V^{-1})
\end{equation}
where 
\begin{equation}\label{Mbullet def}
M^\bullet=R\Gamma(N,T)\oplus \Ind_{N/\Qp}T[0].
\end{equation}
We fix a $\Zp[G]$-projective sublattice $\mathcal L\subseteq \ON$ such that the exponential map $\exp_\calF$ of Lemma \ref{expconvergence}
converges on $\calL^{(r)}$. We set $X(\calL):=\exp_\calF(\calL^{(r)})$ and note that $X(\calL) \sseq \calF(\frp_N^{(r)})$.

The embedding $X(\calL) \xhookrightarrow{f_{\calF,N}} H^1(C_{N, T}^\bullet)$, where $f_{\calF,N}$ is the first map in the exact
sequence (\ref{fFN}), induces an injective map of complexes
$X(\calL)[-1] \lra C_{N,T}^\bullet$. We set
\begin{eqnarray}
  K^\bullet(\calL) &:=& \Ind_{N/\Qp}(T)[0] \oplus X(\calL)[-1], \nonumber \\
  M^\bullet(\calL) &:=& [ \calI_{N/K}(\rhonr) / f_{\calF,N}(X(\calL)) \lra \calI_{N/K}(\rhonr) ] \label{Mbullet}
\end{eqnarray}
with modules in degree $1$ and $2$ and have thus constructed an exact sequence of complexes
\[
  0 \lra K^\bullet(\calL) \lra C_{N, T}^\bullet \oplus \Ind_{N/\Qp}(T)[0] \lra  M^\bullet(\calL) \lra 0.
\]
We first rewrite $C_{N/K}$ in terms of the middle complex and obtain
\[
  C_{N/K}=-\chi_{\ZpG,\BdRG}( C_{N, T}^\bullet \oplus \Ind_{N/\Qp}(T)[0],\exp_V\circ\comp_V^{-1}\circ H^0(\xi))
\]
with $\xi$ as in (\ref{clarification1}). We then use additivity of refined Euler characteristics in distinguished triangles
and derive
\begin{equation}\label{CNK summands}
  C_{N/K} = [X(\calL), \lambda, \Ind_{N/\Qp}(T)] + \chi_{\ZpG,\BdRG}( M^\bullet(\calL), 0),
\end{equation}
where $\lambda$ is the following composite map 
\begin{eqnarray*}
  && X(\calL)_{\BdR} = \calF(\frp_N^{(r)})_{\BdR} \xrightarrow{f_{\calF,N}}
      H^1(C_{N,T}^\bullet)_{\BdR} \\
  && \xrightarrow{H^0(\xi)} H^1(N, T)_{\BdR} \xrightarrow{comp_V\circ\exp_V^{-1}} (\Ind_{N/\Qp}(T))_{\BdR}.
\end{eqnarray*}
Then the term $\chi_{\ZpG,\BdRG}( M^\bullet(\calL), 0)$ is precisely the term which is computed in \cite[Sec.~7.2]{BC2} in the
one-dimensional weakly ramified case.
We will compute this term in arbitrary dimension $r \ge 1$ in Section \ref{tame case} in the tame case and in
Section \ref{weakly ramified case} in the weakly ramified case.

For the first term we obtain
\begin{equation}\label{two summands}
  [X(\calL), \lambda, \Ind_{N/\Qp}(T)] =
  \left[X(\mathcal L),\lambda_2,\bigoplus_{i=1}^r\mathcal L e_i\right]+
  \left[\bigoplus_{i=1}^r\mathcal L e_i,\comp_V,\mathrm{Ind}_{N/\Q_p}T\right]
\end{equation}
where the elements $e_1, \ldots, e_r$ are defined in Lemma \ref{defei} and $\lambda_2$ is the composite map
\begin{equation}\label{def lambda2}
  X(\calL)_{\BdRG} \xrightarrow{H^0(\xi)\circ f_{\calF,N}}  H^1(N, T)_{\BdRG} \xrightarrow{\exp_V^{-1}} t_V(N)_{\BdRG}.
\end{equation}
Note that $e_1, \ldots, e_r$ constitute an $N$-basis of $\DdRN(V) = t_V(N)$.

In the  next three lemmas we will compute the summands in (\ref{two summands}).
  
\begin{lemma}
With $\lambda_2$ denoting the composite map defined in (\ref{def lambda2}) we have
\[
  \left[X(\mathcal L),\lambda_2,\bigoplus_{i=1}^r\mathcal L e_i\right]=0.
\]
\end{lemma}

\begin{proof}
  This proof is an expanded version of the arguments of \cite[p.~509]{IV}.

  Similarly to \cite[p.~360]{BK} we can construct a commutative diagram of exact sequences

\[\xymatrix{0\ar[r]&T\ar[r]^-{\theta_\epsilon^{-1}}\ar[d]^{=}&\varprojlim \calF(\p_{\oo_{\C_p}}^{(r)})\ar[d]^{\theta_\epsilon}\ar[r]&\calF(\p_{\oo_{\C_p}}^{(r)})\ar[d]^{\epsilon\circ\theta_\epsilon}\ar[r]&0\\
0\ar[r]&T\ar[r]\ar[d]^{\mathrm{incl}}&\varprojlim (\oo_{\C_p}^\times)^r(\rhonr)=(R^\times)^r(\rhonr)\ar[r]^-{x\mapsto \epsilon x_0}\ar[d]&(\oo_{\C_p}^\times)^r\ar[r]\ar[d]^{\log_p}&0\\
0\ar[r]&V\ar[r]\ar[d]^=& (\Bcris^{\varphi=p}\cap\BdR^+)^r(\rhonr)\ar[r]^-{\epsilon\circ\theta}\ar[d]^{\mathrm{incl}}&\C_p^r\ar[r]\ar[d]^{v\mapsto t^{-1}\Tnr v}& 0\\
0\ar[r]& V\ar[r]&\Bcris^{\varphi=1}\otimes V\ar[r]& (\BdR/\BdR^+)\otimes V\ar[r]& 0}\]

Note that (differently from \cite{BK}) some of the objects are twisted by $\rhonr$ in order to make all maps $G_N$-invariant;
$\epsilon$ always denotes multiplication by the element $\epsilon\in\overline{\Qpnr}$ which occurs in (\ref{theta iso}).

Taking $G_N$-fixed elements and cohomology we obtain
\[\xymatrix{(\calF(\p_{\oo_{\C_p}}^{(r)}))^{G_N}=\calF(\p_N^{(r)})\ar[r]^-{\partial_{Ku}}\ar[d]^{\log_\calF}&H^1(N,T)\ar[d]^{\mathrm{incl}}\\
(\C_p)^{G_N}=N^r\ar[d]^{s}&H^1(N,V)\ar[d]^{=}\\
((\BdR/\BdR^+)\otimes V)^{G_N}=t_V(N)=\bigoplus_{i=1}^r Ne_i\ar[r]^-{\exp_V}& H^1(N,V),}\]
where the map $s$ is such that $s(v_i)=e_i$ for all $i$.
{Note that this diagram is the higher dimensional version of \cite[(3.4)]{IV}. It makes the identification $s$ of $N^r$ and
  the tangent space $t_V(N)$ explicit.}

We rewrite $\lambda_2$ in terms of the maps in the last diagram and get
\[
  \xymatrix{
    X(\calL)_\BdRG = \calF(\frp_N^{(r)})_\BdRG \ar[rr]^-{H^0(\xi)\circ f_{\calF,N}} \ar[rrrdd]^{\lambda_2}  \ar@/_2pc/[rrr]^\theta && H^1(N, T)_\BdRG \ar[r]^{\partial_{Ku}^{-1}} &
    \calF(\frp_N^{(r)})_\BdRG \ar[d]^{\log_\calF} \\
    & && \BdR^{(r)} \ar[d]^s\\
    & && t_V(N)_\BdRG
  }
\]
By diagram (\ref{clarification2}) we see that $\theta(X(\calL)) = X(\calL)$, so that it remains to show
\[
  [X(\calL), s \circ \log_\calF, \bigoplus_{i=1}^r \calL e_i] = 0,
\]
which is immediate from $\log_\calF(X(\calL)) = \calL^{(r)}$ and $s(v_i)=e_i$ for  $i=1, \ldots,r$.
\end{proof}

\begin{lemma}\label{lemma_diag_6.1IV}
With notation as in (\ref{two summands}) and $m := [K : \Qp]$ we have
\begin{eqnarray*}
  && \left[\bigoplus_{i=1}^r\mathcal L e_i,\comp_V,\mathrm{Ind}_{N/\Qp}T\right] \\
  &=&   -r m  \hat\partial_{\ZpG, \BdRG}^1(t)+
        \left[\bigoplus_{i=1}^r\mathcal L \tilde e_i,\comp_{V(-1)},\mathrm{Ind}_{N/\Qp}T(-1)\right],
\end{eqnarray*}
where the elements $\tilde e_i$ are defined in Lemma \ref{basisV1}.
\end{lemma}

\begin{proof}
  It is easy to see that in the $r$-dimensional setting we also have a diagram as in \cite[(6.1)]{IV}. If we denote the vertical
  maps in this diagram by $f_1$ and $f_2$, then 
  \[
    \begin{split}&\left[\bigoplus_{i=1}^r\mathcal L e_i,\comp_V,\mathrm{Ind}_{N/\Qp}T\right]+
      [\mathrm{Ind}_{N/\Qp}T,f_2,\mathrm{Ind}_{N/\Qp}T(-1)]\\
      &\qquad\qquad=\left[\bigoplus_{i=1}^r\mathcal L e_i,f_1,\bigoplus_{i=1}^r\mathcal L \tilde e_i\right]+
      \left[\bigoplus_{i=1}^r\mathcal L \tilde e_i,\comp_{V(-1)},\mathrm{Ind}_{N/\Qp}T(-1)\right].
\end{split}
\]
Since $f_1$ sends each basis element $e_i$ to $\tilde e_i$, the first summand on the right hand side is trivial.
Both $\mathrm{Ind}_{N/\Qp}T$ and $\mathrm{Ind}_{N/\Qp}T(-1)$ are isomorphic to $\Zp[G]^{rm}$ as $\Zp[G]$-modules.
Via these isomorphisms the map $f_2$ corresponds to multiplication by $t$ and so we obtain
\[
  [\mathrm{Ind}_{N/\Qp}T,f_2,\mathrm{Ind}_{N/\Qp}T(-1)]=
  \left[\Zp[G]^{rm},t,\Zp[G]^{rm}\right]=rm\hat\partial_{\ZpG, \BdRG}^1(t).
\]
\end{proof}

Let $b \in N$ be a normal basis element of $N/K$, i.e. $N = K[G]b$. Let
\[{\rho_b} = \left(\rho_{b, \chi}\right)_{\chi \in \Irr(G)} \in Z(\Qpc[G])^\times = \prod_{\chi \in \Irr(G)} (\Qpc)^\times\]
be defined by 
\[
 \rho_{b, \chi} = \frd_K^{\chi(1)} \calN_{K/\Qp}(b | \chi),
\]
where $\frd_K$ denotes the discriminant of $K/\Qp$ and $\calN_{K/\Qp}(b | \chi)$ the usual norm resolvent, see e.g. \cite[Sec.~2.2]{PickettVinatier}.

We also recall the definition of the twist invariant $\Utwist$ in Section \ref{Utwist def}. The next lemma corrects an error in \cite[Lemma 7.1.2]{BC2} where we just quoted the proof of \cite[Lemma 6.1]{IV}.
  However, whereas we work in the relative group $K_0(\ZpG, \BdRG)$, the authors of loc.cit. work in $K_0(\overline{\Zpnr}[G], \BdRG)$
  where $\Utwist$ vanishes by Remark \ref{Utwist remark}.

\begin{lemma}\label{error lemma}
With $m = [K:\Qp]$ and $\hat\partial^1 = \hat\partial^1_{\ZpG, \BdRG}$ we have
\[
  \left[\bigoplus_{i=1}^r\mathcal L \tilde e_i,\comp_{V(-1)},\mathrm{Ind}_{N/\Qp}T(-1)\right] =
    r\left[\calL, \id, \OKG b \right] +     m \Utwist + r\hat\partial^1(\rho_b).
  \]
\end{lemma}

\begin{proof}
  We let $\Ttriv = \Z_p^{(r)}$ and $\Vtriv=\Q_p^{(r)}$ denote the trivial representations. Let $z_1, \ldots, z_r$ denote
  the canonical $\Zp$-basis of $\Ttriv$.

  In the following we choose to use for each $G_N$-representation $W$
  \[
    \Ind_{N/\Qp}(W) = \{ x\colon G_\Qp \lra W \mid x(\tau\sigma) = \tau x(\sigma)
    \text{ for all } \tau \in G_N, \sigma \in G_\Qp \}
  \]
  as the definition for the induction. Note that if $L/\Qp$ is any field extension (e.g., $L = \BdR$)
  which carries an action of $G_\Qp$
  and $W$ is an $L$-space, then $ \Ind_{N/\Qp}(W)$ is also an $L$-space with $(\alpha x)(\sigma) = \sigma(\alpha) x(\sigma)$
  for all $\alpha \in L$ and $\sigma \in G_\Qp$. We also note that
  \begin{eqnarray*}
    \Ind_{N/\Qp}(W) \lra L[G_\Qp] \tensor_{L[G_N]} W, \quad x \mapsto \sum_{\sigma \in G_\Qp / G_N} \sigma \tensor x(\sigma^{-1}),
  \end{eqnarray*}
  is a well-defined isomorphism of $L[G]$-modules. For the comparison isomorphism $\comp_W$ we then obtain the following simple
  description
  \begin{eqnarray*}
    \comp_W \colon L \tensorQp \left( L \tensorQp W \right)^{G_N} &\lra& \Ind_{N/\Qp}(L \tensorQp W), \\
    l \tensor z &\mapsto& l y_z ,
  \end{eqnarray*}
  where  $l \in L$, $z \in  \left( L \tensorQp W \right)^{G_N}$ and $y_z(\sigma) := z$ for all $\sigma \in G_\Qp$
  (and hence, $(ly_z)(\sigma) = \sigma(l)z$).

  We define a $G$-equivariant isomorphism
  \[
    \tilde h \colon \left( \BdR \tensorQp \Vtriv \right)^{G_N} \lra \left( \BdR \tensorQp V(-1) \right)^{G_N}, \quad
    1 \tensor z_i \mapsto \tilde e_i,
  \]
  and
  \[
    h \colon \Ind_{N/\Qp}\left( \BdR \tensorQp \Vtriv \right) \lra \Ind_{N/\Qp}\left( \BdR \tensorQp V(-1) \right), \quad
    x \mapsto \tilde h \circ x.
  \]
  Then, similar as in the proof of \cite[Lemma 6.1]{IV}, we obtain a commutative diagram
  \[
    \xymatrix{
      \BdR \otimes_\Qp (\BdR \otimes_\Qp \Vtriv)^{G_N} \ar[rr]^-{\comp_{\Vtriv}} \ar[d]^{\BdR \tensor \tilde h} &&
      \mathrm{Ind}_{N/\Qp}(\BdR \otimes_\Qp V_\mathrm{triv}) \ar[d]^{h}\\
      \BdR \otimes_\Qp (\BdR \otimes_\Qp V(-1))^{G_N} \ar[rr]^-{\comp_{V(-1)}} && \mathrm{Ind}_{N/\Qp}(\BdR \otimes_\Qp V(-1)),
    }
  \]
  As a a consequence we derive	
\[\begin{split}
    & \left[\bigoplus_{i=1}^r\mathcal L \tilde e_i,\comp_{V(-1)},\mathrm{Ind}_{N/\Qp}T(-1)\right] \\
    &\quad = \left[\bigoplus_{i=1}^r\mathcal L z_i,\BdR \tensor \tilde h,  \bigoplus_{i=1}^r\mathcal L \tilde e_i \right]
        + \left[\bigoplus_{i=1}^r\mathcal L z_i, \comp_{\Vtriv}, \Ind_{N/\Qp}\Ttriv \right] \\
       &\qquad\qquad + \left[  \Ind_{N/\Qp}\Ttriv , h,  \Ind_{N/\Qp}T(-1) \right] \\
    &\quad =  \left[\bigoplus_{i=1}^r\mathcal L z_i, \comp_{\Vtriv}, \Ind_{N/\Qp}\Ttriv \right]
        + \left[  \Ind_{N/\Qp}\Ttriv , h,  \Ind_{N/\Qp}T(-1) \right] \\
    &\quad =  r \left[ \calL, \id, \OKG b \right] + \left[\bigoplus_{i=1}^r(\OKG b z_i), \comp_{\Vtriv}, \Ind_{N/\Qp}\Ttriv \right] \\
       &\qquad\qquad + \left[  \Ind_{N/\Qp}\Ttriv , h,  \Ind_{N/\Qp}T(-1) \right] 
  \end{split}\]

  The computations in \cite[pages~512-513]{IV} show that 
  \[
    \left[\bigoplus_{i=1}^r\mathcal L z_i, \comp_{\Vtriv}, \Ind_{N/\Qp}\Ttriv \right] =
    r \hat\partial_{\ZpG, \BdRG}^1(\rho).
  \]
  It finally remains to prove that
  \[
    \left[  \Ind_{N/\Qp}\Ttriv , h,  \Ind_{N/\Qp}T(-1) \right] = m \Utwist.
  \]
  To that end we write
  \[
    G_\Qp = \bigcup_{\bar\rho\in G} \bigcup_{\sigma_i\in G_K\backslash G_\Qp} G_N\rho\sigma_i
  \]
  and define elements $x_{ij} \in \Ind_{N/\Qp}(\Ttriv)$ and  $y_{ij} \in \Ind_{N/\Qp}(T(-1))$ for
  $i = 1, \ldots,m$ and $j = 1, \ldots, r$ by
  \[
    x_{ij}(\rho\sigma_k) =
    \begin{cases}
      z_j, & \text{ if } i=k \text{ and } \bar\rho = 1, \\
      0, & \text{ otherwise},
    \end{cases}
  \]
  and
  \[
    y_{ij}(\rho\sigma_k) =
    \begin{cases}
      \tilde v_j, & \text{ if } i=k \text{ and } \bar\rho = 1, \\
      0, & \text{ otherwise}.
    \end{cases}
  \]
  Then the $x_{ij}$, respectively, the $y_{ij}$, constitute a $\ZpG$-basis of $\Ind_{N/\Qp}(\Ttriv)$, respectively,
  $\Ind_{N/\Qp}(T(-1))$.

  For $1 \le i,k \le m$ and $1 \le j \le r$ we compute
  \[
    \left( h(x_{ij}) \right) (\rho\sigma_k) =
    \begin{cases}
      \tilde e_j, & \text{ if } i=k \text{ and } \bar\rho = 1, \\
      0, & \text{ otherwise}.
    \end{cases}
  \]
  For $1 \le s \le m$ and $1 \le t \le r$ let $\xi_{st}$ be indeterminates ( with values in $\BdR$). Then
  \[
    \left( \sum_{s,t} \xi_{st} y_{st} \right) (\rho\sigma_k) =
    \begin{cases}
      \sum_t (\rho\sigma_k)(\xi_{kt}) \tilde v_t, & \text{ if } \bar\rho = 1, \\
      0, & \text{ otherwise}.
    \end{cases}
  \]
  Since $\tilde e_j = \sum_t \Tnr_{tj} \tensor \tilde v_t$ we obtain
  \[
   \sum_t (\rho\sigma_k)(\xi_{kt}) \tilde v_t = 
   \begin{cases}
     \sum_t \Tnr_{tj} \tensor \tilde v_t, &  \text{ if } i=k, \\
     0, &  \text{ if } i \ne k,
   \end{cases}
 \]
 and hence
 \[
   \xi_{it} = \sigma_i^{-1}(\Tnr_{tj}), \quad \xi_{it} = 0 \text{ for } i \ne k.
 \]
 We conclude that the transition matrix is a block matrix of the form
 \[
   C = \left(
     \begin{array}{ccc}
       \sigma_1^{-1}(\Tnr) & & \\ & \ddots & \\ & & \sigma_m^{-1}(\Tnr) 
     \end{array}
   \right).
 \]
 We recall that $\varphi(\Tnr) = u^{-1}\Tnr$ and fix $\alpha_i \in \hat\Z$
 such that $\sigma_i |_\Qpnr = \varphi^{\alpha_i}$. Note that for all $n\in\Z$, $\varphi^{-n}(\Tnr)\cdot (\Tnr)^{-1}=u^n$, which coincides with $\rhonr(\varphi^n)$. By a continuity $\varphi^{-\alpha_i}(\Tnr)\cdot (\Tnr)^{-1}=u^{\alpha_i}$, where $u^{\alpha_i}$
   is well-defined by \cite[Lemma~1.5]{Cobbe18}. Then
 \[
   \det(C) = \prod_{i=1}^m \det\left( u^{\alpha_i}\Tnr \right) =
   \det(u)^\alpha \det(\Tnr)^m
 \]
 with  $\alpha = \sum_i \alpha_i$. Since $\det(u)^\alpha \in \Z_p^\times \sseq \ZpG^\times$ the result follows
 from the definition of $\Utwist$.

\end{proof}
  
We summarize  the results of the previous lemmas in the following proposition.

\begin{prop}\label{CNKgeneral}
With $\hat\partial^1 = \hat\partial_{\ZpG, \BdRG}^1$ and $\chi = \chi_{\ZpG, \BdRG}$ we have
\[
  C_{N/K} =  -rm\hat\partial^1(t)+ r\left[\calL, \id, \OKG b\right] + r\hat\partial^1(\rho_b)+
    m \Utwist - \chi(M^\bullet(\calL), 0).
\]
\end{prop}
  \begin{remark}\label{Remark Error}
    To compare Proposition \ref{CNKgeneral} and \cite[(55)]{BC2} we first note that in loc.cit. we have $\calL = \OKG b$.
    The additional new term $\Utwist$ emerges from the computations in Lemma \ref{error lemma}.
    The error does not affect the validity of any of the arguments in \cite{BC2}, it just forces us to adapt our definition of
    $R_{N/K}$ and $\tilde{R}_{N/K}$.
  \end{remark}

To finish the proof of the conjecture it is necessary to compute explicitly the
term $\chi_{\Z_p[G], \BdR[G]}(M^\bullet(\calL), 0)$. For this we will consider the tame and the weakly ramified case separately.

\end{subsection}

\begin{subsection}{The tame case}\label{tame case}
  In this subsection we let  $N/K$ be tame und compute the term $\chi_{\ZpG, \BdRG}(M^\bullet(\calL), 0)$ from
  (\ref{CNK summands}).
  In the tame case, by results of Ullom, we can and will use $\calL = \frp_N^\nu$ for a large enough
  positive integer $\nu$ and we also fix $b \in \ON$ such that $\ON=\OKG b$.

 \begin{prop}\label{tameH1H2}
We have
\[
  \begin{split}&\chi_{\Z_p[G], \BdR[G]}(M^\bullet(\calL), 0)=
    r[\mathfrak p_N^{\nu},\id,\mathfrak p_N]-\hat\partial_{\ZpG, \BdRG}^1({}^*(\det(1-u^{d_K}F^{-1})e_I)).
  \end{split}
\]
\end{prop}

\begin{proof}
  The key point in the proof is that by Proposition \ref{tame cohomology}
  the cohomology modules of $M^\bullet(\calL)$ are perfect,
  so that we can compute the refined Euler characteristic of $M^\bullet(\calL)$ in terms of cohomology
  without explicitly using the complex. In a little more detail, we note that the mapping cone of
  \[
    \calF(\frp_N^{(r)}) / X(\mathcal L)[1] \to M^\bullet(\calL),
  \]
  where the map in degree $1$ is induced by $f_{\calF,N}$, is isomorphic to $H^2(N, T)[2]$.
  We also recall from Section \ref{sec cohomology} that we identify $H^2(N, T)$ with $\calZ / (F_N -1)\calZ$.
  Hence we conclude from \cite[Thm.~5.7]{BrBuAdditivity} that
  \[
    \begin{split}
      &\chi_{\Z_p[G], \BdR[G]}(M^\bullet(\calL), 0)\\
      &\qquad=\chi_{\Z_p[G], \BdR[G]}\left( \calF(\frp_N^{(r)})/X(\mathcal L)[1],0 \right) +
              \chi_{\Z_p[G], \BdR[G]}(\calZ / (F_N -1)\calZ, 0).
    \end{split}
  \]
  To compute the first summand we observe that by Proposition \ref{log exp iso}
  we have $X(\calL) = \calF\left( (\frp_N^\nu)^{(r)}\right)$. Since for each integer $i \ge 0$
  \begin{equation}\label{fil iso}
    \calF\left( (\frp_N^i)^{(r)}\right) / \calF\left( (\frp_N^{i+1})^{(r)}\right) \cong
    (\frp_N^i)^{(r)} / (\frp_N^{i+1})^{(r)}
  \end{equation}
  a standard argument shows that
  \[
    \chi_{\Z_p[G], \BdR[G]}\left( \calF(\frp_N^{(r)})/X(\mathcal L)[1],0 \right) =  
    \chi_{\Z_p[G], \BdR[G]}\left( \frp_N^{(r)} / (\frp_N^\nu)^{(r)} [1],0 \right) = r [\frp_N^\nu, \id, \frp_N].
  \]
  For the computation of the second term we consider the short exact sequence of $G$-modules
  \[
    0 \to \Z_p^r[G/I] \xrightarrow{F^{-1}u^{d_K}-1} \Z_p^r[G/I] \xrightarrow{\pi} \calZ / (F_N -1)\calZ \to 0,
  \]
{where $\pi(z(\bar g))=g\cdot (z+(F_N-1)\calZ)$ for all $z\in\Z_p^r$ and $g\in G$ and where $G$ acts on $\calZ / (F_N -1)\calZ$ through any lift of its elements to $G_K$ (which is well-defined since elements of $G_N$ act trivially).}

Let $x=\sum_{i=0}^{d-1}\alpha_i F^{-i}\in \Z_p^r[G/I]$ be an element in the kernel of the map on the left. Then $u^{d_K}\alpha_{i-1}-\alpha_i=0$ for all $i$. Hence $(u^{dd_K}-1)\alpha_i=0$ and by the assumption $\mathcal Z^{G_N}=1$, it follows that $\alpha_i=0$ for all $i$. Hence the map on the left is injective.

Next we see that $\pi((F^{-1}u^{d_K}-1)e_i)=(\rhonr(F^{-1})u^{d_K}-1)e_i=0$.

Conversely let $x=\sum_{i=0}^{d-1}\alpha_i F^{-i}\in\Z_p^r[G/I]$ be such that $\pi(x)=0$. Modulo the image of $F^{-1}u^{d_K}-1$, $x$ has a representative $y\in\Z_p^r$. We must show that $y\in \mathrm{im}(F^{-1}u^{d_K}-1)$. Since $\pi(y)=0$, there exists $z\in\Z_p^r$ such that $y=(u^{dd_K}-1)z=((F^{-1}u^{d_K})^d-1)z$, which is in the image of $F^{-1}u^{d_K}-1$. Hence we have exactness in the middle term.

To prove the exactness of the sequence, it remains to check the surjectivity of the map on the right, which is obvious.

Since we are considering the case of a tame extension, $\Z_p^r[G/I]$ is a projective $\Z_p[G]$-module and we have:
\[\chi_{\Z_p[G], \BdR[G]}(H^2(N,T)[2],0)=-[\Z_p^r[G/I],F^{-1}u^{d_K}-1,\Z_p^r[G/I]].\]
The results follows.

\end{proof}

\end{subsection}

\begin{subsection}{The weakly ramified case}\label{weakly ramified case}

  In this subsection we let $p$ be an odd prime. Let $K/\Qp$ be the unramified extension of degree $m$. We let $N/K$ be a weakly and
  wildly ramified finite abelian extension with cyclic ramification group. We let $d = d_{N/K}$ be  the inertia degree of $N/K$ and assume
  that $m$ and $d$ are relatively prime.

  The aim of this subsection is to compute the term $\chi_{\ZpG, \BdRG}(M^\bullet(\lambda), 0)$ from (\ref{CNK summands})
  in this weakly and wildly ramified situation. For that purpose we aim to generalize the methods of \cite{BC2}, however,
  this forces us to introduce a further technical condition which might be either\\

  \noindent{\bf Hypothesis (T):} $U_N \equiv 1 \pmod{p}$ \\

or \\

  \noindent{\bf Hypothesis (I):} $U_N - 1$ is invertible modulo $p$ \\

  Here (T) stands for trivial reduction modulo $p$ and (I) for invertible modulo $p$.
    If we set $\omega := v_p(\det(U_N-1))$ then we have the following equivalences
  \[
    (I) \text{ holds } \iff U_N - 1 \in \Gl_r(\Zp) \iff \omega = 0
  \]
  Note also that Hypothesis (T) immediately implies $\omega > 0$. However, in the higher dimensional setting
  there are mixed cases, where none of our hypotheses holds.

As in \cite{BC2} we have a diagram of fields as follows.
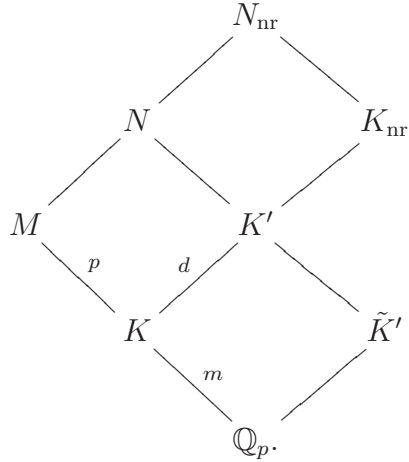
\[\xymatrix{
&&\Nur\ar@{-}[dr]\ar@{-}[dl]\\
&N\ar@{-}[dr]\ar@{-}[dl]&&\Kur\ar@{-}[dl]\\
M\ar@{-}[dr]^p&&K'\ar@{-}[dl]_d\ar@{-}[dr]\\
&K\ar@{-}[dr]^m&&\tilde K'\ar@{-}[dl]\\&&\Q_p.}\]
Here $K'/K$ is the maximal unramified subextension of $N/K$,  $M/K$ is a weakly and wildly ramified cyclic extension of degree $p$
and $N = MK'$. Since $\gcd(m,d)=1$, there exists $\tilde K'/\Q_p$ of degree $d$ such that $K'=K\tilde K'$.

The following lemma generalizes \cite[Lemma~7.2.1]{BC2}.

\begin{lemma}\label{projectivity}
For $n \ge 2$ one has
\[\calF((\p_N^n)^{(r)}) \text{ is }\ZpG \text{-projective}\iff n\equiv 1\pmod{p}.\]
Moreover,
\[\calF(\p_N^{(r)}) \text{ is }\ZpG \text{-projective} \iff \text{Hypothesis (I) holds}.\]
\end{lemma}

\begin{proof}
  For $n \ge 2$ the formal logarithm induces an isomorphism $\calF((\p_N^n)^{(r)}) \cong (\p_N^n)^{(r)}$ of $\ZpG$-modules by
  Proposition \ref{log exp iso}. Hence the first assertion follows from \cite[Th.~1.1 and Prop.~1.3]{Koeck}.

  We henceforth assume $n=1$. By Lemma \ref{ct coh} we know that $\calF(\p_N^{(r)})$ is cohomologically trivial,
    if and only if Hypothesis (I) holds.
    Hence it suffices to prove that $\calF(\p_N^{(r)})$ is torsion-free.
    By \cite[Lemma 2.1]{Cobbe18} the module $\calF(\p_N^{(r)})$ is isomorphic to $\left( \prod_r \widehat{N_0^\times} (\rhonr)\right)^{G_N}$
    which is torsion-free. Indeed, any tuple $(\zeta_1, \ldots, \zeta_r)$ of $p$-th roots of unity $\zeta_1, \ldots, \zeta_r \in
    \widehat{N_0^\times} \cong \Zp \times U_{N_0}^{(1)}$ must be contained in $N$ (because $N_0/N$ is unramified).
    Hence $(\zeta_1, \ldots, \zeta_r)$ is fixed by $G_N$, if and only if it is fixed by $F_N$, if and only if $\zeta_1 = \ldots = \zeta_r = 1$ (using Hypothesis (I)).

\end{proof}

By Lemma \ref{projectivity}  we can and will take $\calL = \frp_N^{p+1}$ and thus obtain
  $$X(\calL) = \calF\left( (\frp_N^{p+1})^{(r)} \right).$$
  
We recall some of the notations from \cite{BC2}. We put $q=p^m$, $b=F^{-1}$ and consider an element
$a\in \Gal(\Nur/K)$ such that $\Gal(M/K)=\langle a|_M\rangle$, $a|_{\Kur}=1$.
Since there will be no ambiguity, we will denote by the same letters $a,b$ their restrictions to $N$.
Then $\Gal(N/K) = \langle a,b \rangle$ and $\mathrm{ord}(a)=p, \mathrm{ord}(b) = d$. We also define $\T_a := \sum_{i=0}^{p-1}a^i$.

Let $\theta_1\in M$ be such that $\T_{M/K}\theta_1=p$ and $\oo_K[\Gal(M/K)]\theta_1=\p_M$.
Let $\theta_2$ (resp. $A$) be a normal integral basis generator of trace one for the extension
$\tilde K'/\Q_p$ (resp. $K/\Q_p$). Let $\alpha_1\in\oo_K^\times$ be such that $\theta_1^{a-1}\equiv 1-\alpha_1\theta_1\pmod{\p_M^2}$.
If we set
\begin{equation}\label{choice of alpha}
\alpha_1,\alpha_2=\alpha_1A,\alpha_3=\alpha_1A^{\varphi},\dots, \alpha_m=\alpha_1A^{\varphi^{m-2}},
\end{equation}
then these elements form a $\Zp$-basis of $\oo_K$ (see \cite[(60)]{BC2}).

Furthermore, we use \cite[Lemma~2.4]{Cobbe18} to find for $i=1, \ldots,r$ an element  $\gamma_i\in \prod_r U_{N_0}^{(1)}$ such that
\[
  (F_N-1)\cdot\gamma_i =  (\theta_1)_i,
\]
where
\[
  (\theta_1)_i := (1,\dots,1,\theta_1^{a-1},1,\dots,1)
\]
with the non-trivial entry is the $i$-th component.

Let
\[
  W'=\Z_p[G]^r z_1\oplus \Z_p[G]^r z_2,
\]
\[
  W_{\geq n}=\bigoplus_{j=n}^{p-1}\bigoplus_{k=1}^{m}\Z_p[G]^r v_{k,j} \cong
  \bigoplus_{j=n}^{p-1}\bigoplus_{k=1}^m\Z_p[G]^r\alpha_kw_j =
  \bigoplus_{j=n}^{p-1}\OKG^r w_j
\]
and put 
\[
  W=W_{\geq 0}.
\]
If we write $e_1, \ldots, e_r$ for the standard $\Zp$-basis of $\Zp^{(r)}$, then a general element of $W$ is of the form
\[
  \sum_{i=1}^r \sum_{j=1}^{p-1} \sum_{k=1}^m \lambda_{i,j,k}e_i v_{k,j} = \sum_{i=1}^r \sum_{j=1}^{p-1} \mu_{i,j,k}e_i w_{j}
\]
with $\lambda_{i,j,k} \in \ZpG$ or $\mu_{i,j} \in \OKG$. We will apply this convention analogously for the modules $W'$ and
$ W_{\geq n}$.

We define a matrix $E \in M_r(\calO_{K'})$ by
\begin{equation}\label{E def}
E=  \begin{cases}1&\text{under Hypothesis (I)}\\
\sum_{i=0}^{dm-1}(A\theta_2)^{\varphi^i}u^{-i}&\text{under Hypothesis (T)}
\end{cases}
\end{equation}
and a matrix
\begin{equation}\label{tilde u def}
\tilde u=  \begin{cases}1&\text{under Hypothesis (I)}\\
u&\text{under Hypothesis (T)}
\end{cases}
\end{equation}

We also recall that in \cite[Lemma~1.9]{Cobbe18} we constructed an element
$\epsilon\in \Gl_r(\overline{\Zpnr})$ such that $u=\epsilon^{-1}\cdot\varphi(\epsilon)$.

\begin{lemma}\label{E lemma}
  The following assertions hold:
\begin{enumerate}[(a)]
\item $E\in \Gl_r(\oo_{K'})$.
\item $\varphi(E)\equiv \tilde u E \equiv E\tilde u \pmod{p\oo_{K'}}$.
\item $\varphi(\epsilon^{-1} E)\equiv u^{-1}\tilde u \epsilon^{-1} E \pmod{p}$.
\end{enumerate}
\end{lemma}

\begin{proof}
  Let $H=\mathrm{Gal}(K'/\Q_p)$ and let $f \colon H\to K'$ be defined by $f(\sigma)=(A\theta_2)^{\sigma^{-1}}$.
  Then, applying \cite[Lemma 5.26 (a)]{Washington}, we obtain
  \[
    \det((A\theta_2)^{\tau\sigma^{-1}})_{\sigma,\tau\in H}=\prod_{\chi\in \hat H}\sum_{\sigma\in H}(A\theta_2)^{\sigma^{-1}}\chi(\sigma).
  \]
  Since $A\theta_2$ is an integral normal basis generator and $K'/\Q_p$ is unramified,
  the left hand side $\det(\tau\sigma^{-1}(A\theta_2))_{\sigma,\tau\in H}$ is a unit (whose square is the discriminant of $K'/\Q_p$).
  Therefore, for each character $\chi\in \hat H$ the factor 
  \[
    \sum_{\sigma\in H}(A\theta_2)^{\sigma^{-1}}\chi(\sigma)=\sum_{i=0}^{dm-1}(A\theta_2)^{\varphi^{i}}\chi(\varphi^{-i})
  \]
  is a unit and hence $\sum_{i=0}^{dm-1}(A\theta_2)^{\varphi^{i}}\varphi^{-i}$ is a unit
  in the maximal order $\calM$ of $K'[H]$. Since $\sum_{i=0}^{dm-1}(A\theta_2)^{\varphi^{i}}\varphi^{-i}\in \mathcal O_{K'}[H]$,
  we deduce from the well-known fact
  \[
    \calM^\times \cap \calO_{K'}[H] = \calO_{K'}[H]^\times
  \]
  that $\sum_{i=0}^{dm-1}(A\theta_2)^{\varphi^{i}}\varphi^{-i}\in \mathcal O_{K'}[H]^\times$. We now apply the character $\rhonr_\Qp$ and easily derive (a).

For the proof of (b) we can assume Hypothesis (T) and we compute
\[
  \varphi(E) = \sum_{i=0}^{dm-1} (A\theta_2)^{\varphi^{i+1}}u^{-(i+1)} u \equiv Eu \equiv uE \pmod{p\oo_{K'}}
\]
where the congruences hold because we have $u^{md} \equiv 1\pmod{p}$ by hypothesis (T).

If Hypothesis (T) holds, then part (c) is then an immediate  consequence of \cite[Lemma~1.9]{Cobbe18}. Unter Hypothesis (I) it follows from the definitions.
\end{proof}

Generalizing the approach of \cite[Sec.~6.3]{BC2} we define a $\ZpG$-module homomorphism
\[
  \tfW\colon W \lra \calF(\frp_N^{(r)})
\]
by
\[
  \tfW(e_iv_{k,j})=E\alpha_k(a-1)^j\theta e_i = E \left(
    \begin{array}{c}
      0 \\ \vdots \\ 0 \\ \alpha_k (a-1)^j\theta \\ 0 \\ \vdots \\ 0 
    \end{array}
  \right)
\]
for all $i,j,k$, where $\theta=\theta_1\theta_2$.
We denote by $\fW$ the composition of $\tfW$ with the projection to $\calF(\p_N^{(r)})/\calF((\p_N^{p+1})^{(r)})$.

In order to generalize \cite[Lemma~7.2.4]{BC2}, we first need a higher dimensional version of \cite[Lemma~4.1.7]{BC}.
\begin{lemma}
For $\nu\in \oo_K[G]$, $i=1,\dots,r$, $k=1,\dots,m$ and $j=0,\dots,p-1$ we have
\[\fW(\nu e_i\alpha_k w_j)\equiv \nu E\alpha_k (a-1)^j\theta e_i\pmod{\p_N^{j+2}}.\]
\end{lemma}

\begin{proof}
We write $\nu=\sum_{h,l,i,k}\nu_{h,l,i,k}a^hb^le_i\alpha_k w_j$ for some $\nu_{h,l,i,k}\in\Z_p$. Then
\[\fW\left(\sum_{h,l,i,k}\nu_{h,l,i,k}a^hb^le_i\alpha_k w_j\right)=\sum_{h,l,i,k}\nu_{h,l,i,k}a^hb^l \cdot E \alpha_k (a-1)^j \theta e_i.\]
Using the isomorphism in (\ref{fil iso}), this is congruent to
\[\sum_{h,l,i,k}\nu_{h,l,i,k}a^hb^l E \alpha_k (a-1)^j \theta e_i \pmod{\p_N^{j+2}}.\]
\end{proof}

\begin{lemma}\label{f3surj}
The map $\fW$ is surjective. More precisely, for $j\geq 0$, 
\[
  \fW(W_{\geq j})=\calF((\p_N^{j+1})^{(r)})/\calF((\p_N^{p+1})^{(r)}).
\]
\end{lemma}

\begin{proof}
For $j=p$, $W_{\geq p}=\{0\}$ and $\calF((\p_N^{p+1})^{(r)})/\calF((\p_N^{p+1})^{(r)})=\{0\}$, so the  result is trivial.

We assume the result for $j+1$ and proceed by descending induction. Let $x\in \calF((\p_N^{j+1})^{(r)})$. As in the proof of \cite[Lemma~7.2.4]{BC2}, there exist $\nu_{h,\ell,i,k}\in\Z_p$ such that
\[\begin{split}x&\equiv E\sum_{h,\ell,i,k} \nu_{h,\ell,i,k}a^hb^\ell\alpha_k(a-1)^j\theta e_i\pmod{\p_N^{j+2}}\\
&\equiv \sum_{h,\ell,i,k}\tilde u^{m\ell} E^{\varphi^{-m\ell}} \nu_{h,\ell,i,k}a^hb^\ell\alpha_k(a-1)^j\theta e_i\pmod{\p_N^{j+2}}\\
&\equiv \sum_{h,\ell,i,k}\tilde u^{m\ell} \nu_{h,\ell,i,k}a^hb^\ell E \alpha_k(a-1)^j\theta e_i\pmod{\p_N^{j+2}}\\
&\equiv \tfW\left(\sum_{h,\ell,i,k}\tilde u^{ml} \nu_{h,\ell,i,k}a^hb^\ell e_i\alpha_k w_j\right)\pmod{\p_N^{j+2}}
.\end{split}\]
This means that the class $\pi(x)$ of $x$ in $\calF(\p_N^{(r)})/\calF((\p_N^{p+1})^{(r)})$ is the sum of an element in the
image of $W_{\geq j}$ and an element in $\calF((\p_N^{j+2})^{(r)})/\calF((\p_N^{p+1})^{(r)})$,
which is by assumption in the image of $W_{\geq j+1}\subseteq W_{\geq j}$.
\end{proof}

\begin{lemma}\label{defsjk}
Let $1\leq i\leq r$, $0\leq j\leq p-1$, $1\leq k\leq m$. Then there exists $\mu_{i,j,k}\in W_{\geq j+2}$ such that the element
\[s_{i,j,k}=\alpha_k(a-1) e_iw_j-\alpha_k e_i w_{j+1}+\mu_{i,j,k}\]
is in the kernel of $\fW$. Here $w_p$ should be interpreted as $0$.
\end{lemma}

\begin{proof}
As in the proof of \cite[Lemma~7.2.5]{BC2}, we see that the formal subtraction of $X$ and $Y$ takes the form
\[X-_\calF Y=X-Y+(X^tA_hY)_h-(Y^tA_hY)_h+(\deg \ge 3).\]
with $A_h\in M_r(\Z_p)$ for $h=1,\dots,r$. In analogy to the proof of \cite[Lemma~7.2.5]{BC2} we set
\[x := E\alpha_k (a-1)^ja\theta e_i, \quad y := E\alpha_k(a-1)^j\theta e_i, \quad z := x-y=E\alpha_k(a-1)^{j+1}\theta e_i\]
and we obtain that $x -_\calF y -_\calF z \equiv  0\pmod{\calF((\p_N^{j+3})^{(r)})}$. Therefore
\[\tfW(\alpha_k(a-1) e_i w_j-\alpha_k e_i w_{j+1})\equiv 0\pmod{\calF((\p_{N}^{j+3})^{(r)})}\]
and we conclude the proof of the lemma using Lemma \ref{f3surj}.
\end{proof}

\begin{lemma}\label{defrk}
The elements
\[r_{i,1}=\T_a e_i \alpha_1 w_0+\epsilon u^{-1}\tilde u^{1-m\tilde m}\epsilon^{-1}b^{-\tilde m} e_i \alpha_{1} w_{p-1} -e_i \alpha_1 w_{p-1},\]
\[r_{i,k}=\T_a e_i \alpha_k w_0+\epsilon u^{-1}\tilde u^{1-m\tilde m}\epsilon^{-1}b^{-\tilde m} e_i \alpha_{k+1} w_{p-1} -e_i \alpha_k w_{p-1},\]
\[r_{i,m}=\T_a e_i \alpha_m w_0+\epsilon u^{-1}\tilde u^{1-m\tilde m}\epsilon^{-1}b^{-\tilde m} e_i \left(\alpha_1-\sum_{i=2}^m\alpha_i\right)w_{p-1}-e_i \alpha_m w_{p-1}\]
for $1\leq i\leq r$ and $1< k<m$ are in the kernel of $f_4$. Note that $\epsilon u^{-1}\tilde u^{1-m\tilde m}\epsilon^{-1}$ has coefficients in $\Zpnr$ and is fixed by $\varphi$, hence it has coefficients in $\Zp$.
\end{lemma}

\begin{proof}

We denote by $v_j$ the $j$-th component of a vector $v$. Using \cite[Lemma~3.2.2]{BC} we calculate
\[\begin{split}(\norm_{N_0/K_0}(\epsilon^{-1}E\alpha_k\theta e_i))_j&=\norm_{N_0/K_0}(\theta_1){(\epsilon^{-1}E\alpha_k\theta_2 e_i)_j}^p\\
&\equiv -\alpha_k^p\theta_2^p\alpha_1^{1-p}p{(\epsilon^{-1}E e_i)_j}^p\pmod{\p_{N_0}^{p+1}}\\
&\equiv -\alpha_k^p\theta_2^p\alpha_1^{1-p}p(\varphi(\epsilon^{-1}E) e_i)_j\pmod{\p_{N_0}^{p+1}}\\
&\equiv -\alpha_k^p\theta_2^p\alpha_1^{1-p}p \left(u^{-1}\tilde u\epsilon^{-1}E e_i\right)_j\pmod{\p_{N_0}^{p+1}}.\end{split}\]
We also compute
\[\trace_{N_0/K_0}(\epsilon^{-1}E\alpha_k\theta e_i)=\trace_{N_0/K_0}(\theta_1)\epsilon^{-1}E\alpha_k\theta_2 e_i\equiv p\alpha_k\theta_2 \epsilon^{-1}E  e_i\pmod{\p_{N_0}^{p+1}}.\]
With this in mind, we can do analogous calculations to those in \cite[Lemma~4.2.6]{BC} and \cite[Lemma~7.2.6]{BC2} and obtain
\[f_4(\T_a e_i \alpha_k w_0)\equiv 1+\left(\alpha_k\theta_2-\alpha_{1}\left(\frac{\alpha_k}{\alpha_1}\right)^p\theta_2^{b^{-\tilde m}}u^{-1}\tilde u\right)p\epsilon^{-1}E e_i\pmod{\p_{N_0}^{p+1}}.\]

Furthermore,
\[\begin{split}&f_4(\epsilon u^{-1}\tilde u^{1-m\tilde m}\epsilon^{-1}b^{-\tilde m}e_i\alpha_{k+1}w_{p-1})\\
&\qquad\equiv 1+u^{-1}\tilde u^{1-m\tilde m}\epsilon^{-1}E^{\varphi^{m\tilde m}}\alpha_{k+1}p\theta_2^{b^{-\tilde m}} e_i\pmod{\p_{N_0}^{p+1}}\\
&\qquad\equiv 1+u^{-1}\tilde u^{1-m\tilde m}u^{m\tilde m}(\epsilon^{-1}E)^{\varphi^{m\tilde m}}\alpha_{k+1}p\theta_2^{b^{-\tilde m}} e_i\pmod{\p_{N_0}^{p+1}}\\
&\qquad\equiv 1+u^{-1}\tilde u \epsilon^{-1}E\alpha_{k+1}p\theta_2^{b^{-\tilde m}} e_i\pmod{\p_{N_0}^{p+1}}.
\end{split}\]
It is straightforward to adapt the remaining calculations from \cite[Lemma~7.2.6]{BC2} and conclude that $r_{i,k}\in \ker f_4$ for for $1\leq i\leq r$ and $1< k<m$. The proof that $r_{i,m}\in \ker f_4$ is analogous.
\end{proof}

From now on we have to distinguish the cases of Hypothesis (T) and (I).

We start assuming Hypothesis (I).

Following the computations in \cite[Sect.~7.3]{BC2}, we have
\[\chi_{\ZpG, \BdR[G]}(M^\bullet(\calL), 0) = [\calF((\p_N^{p+1})^{(r)}), \id, \calF(\p_N^{(r)})]=[\ker(\fW), \id, W].\]

\begin{lemma}\label{genkerh}
The $pmr$ elements $r_{k},s_{j,k}$ for $0\leq j\leq p-2$, $1\leq k\leq m$ constitute a $\ZpG$-basis of $\ker \fW$.
\end{lemma}

\begin{proof}
We adapt the proof of \cite[Lemmas~7.3.1]{BC2}. We write the coefficients of the  $e_i\alpha_kw_{p-1}$-components, $i=1,\dots,r$, $k=1,\dots,m$, of the elements $r_{i,j}$, $j=1,\dots,m$, into the columns of an $mr \times mr$  matrix which we call $\calM$ and whose entries are $r\times r$ blocks,
\[\calM=\begin{pmatrix}
\epsilon u^{-1}\epsilon^{-1} b^{-\tilde m}\!-\!1\!\!\!\!\!\!&0&\cdots&0&0&\epsilon u^{-1}\epsilon^{-1} b^{-\tilde m}\\
0&-1&\cdots&0&0&-\epsilon u^{-1}\epsilon^{-1}b^{-\tilde m}\\
0&\!\!\!\epsilon u^{-1}\epsilon^{-1}b^{-\tilde m}\!\!\!&\cdots&0&0&-\epsilon u^{-1}\epsilon^{-1}b^{-\tilde m}\\
\vdots&\vdots&\ddots&\vdots&\vdots&\vdots\\
0&0&\cdots&\!\!\epsilon u^{-1}\epsilon^{-1}b^{-\tilde m}\!\!\!\!\!\!\!&-1&-\epsilon u^{-1}\epsilon^{-1}b^{-\tilde m}\\
0&0&\cdots&0&\epsilon u^{-1}\epsilon^{-1}b^{-\tilde m}&-1\!-\!\epsilon u^{-1}\epsilon^{-1} b^{-\tilde m}
\end{pmatrix}\!.\]
By Lemma \ref{detUcris} and analogous computations as in \cite{BC2} we obtain 
\[\det\calM=(-1)^{r(m-1)}\det(u^{-m}b^{-1}-1).\]
The rest of the proof works exactly as in \cite{BC2}; note that Hypothesis (I) plays the role of the assumption $\omega=0$ in the one-dimensional setting.
\end{proof}

\begin{prop}\label{prop 117}
Assume Hypothesis (I). For $\calL = \p_N^{p+1}$ the element 
\[\chi_{\ZpG, \BdR[G]}(M^\bullet(\calL), \lambda_2^{-1}) \in K_0(\ZpG, \BdR[G])\]
is contained in $K_0(\ZpG, \QpG)$ and represented by $\epsilon \in \QpG^\times$ where
\[\begin{split}
\epsilon_{\chi\phi}^{}&=\begin{cases}p^{mr}&\text{if $\chi=\chi_0$}\\
(-1)^{r(m-1)} \det(u^{-m}\phi(b)^{-1} - 1) (\chi(a)-1)^{mr(p-1)}&\text{if $\chi\neq\chi_0$.}\end{cases}\\
\end{split}\]
\end{prop}
\begin{proof}
We choose the elements $r_{i,k}$ and $s_{i,j,k}$ of Lemma \ref{genkerh} as a $\ZpG$-basis of $\ker(\fW)$ and fix the canonical 
$\ZpG$-basis of $W$. Then
\[\chi_{\ZpG, \BdR[G]}(M^\bullet(\calL), \lambda_2^{-1}) = [\ker(\fW), \id, W]\] 
is represented by the determinant of
\[\mm=\begin{pmatrix}
\T_aI&(a-1)I&0&\cdots&0&0\\
0&-I&(a-1)I&\cdots&0&0\\
0&*&-I&\cdots&0&0\\
\vdots&\vdots&\vdots&\ddots&\vdots&\vdots\\
0&*&*&\cdots&-I&(a-1)I\\
\calM&*&*&\cdots&*&-I
\end{pmatrix},\]
where in the above $pmr\times pmr$ all the entries are $mr\times mr$-blocks. Recalling that $p$ is odd, we get:
\[\begin{split}\det(\chi\phi(\mm))&=\begin{cases}p^{mr}(-1)^{mr(p-1)}&\text{if $\chi=\chi_0$}\\(-1)^{m^2r^2(p-1)}\det(\chi\phi(\calM))(\chi(a)-1)^{mr(p-1)}&\text{if $\chi\neq\chi_0$}\end{cases}\\
&=\begin{cases}p^{mr}&\text{if $\chi=\chi_0$}\\ (-1)^{r(m-1)}\det(u^{-m}\phi(b)^{-1}-1)(\chi(a)-1)^{mr(p-1)}&\text{if $\chi\neq\chi_0$.}\end{cases}\end{split}\]
\end{proof}

It remains to be considered the case of Hypothesis (T).

We now proceed as in \cite[Section~7.4]{BC2}.
\begin{lemma}\label{deff4}
There is a commutative diagram of $\Z_p[G]$-modules with exact rows
\[
  \xymatrix@C-=0.5cm{
    0\ar[r]& X(2)\oplus W\ar[r]\ar[d]^-{\tilde f_4}&W'\oplus W\ar[rr]^-{\delta_2}\ar[d]^-{\tilde f_3}&&\Z_p[G]^r z_0
    \ar[r]^-{\pi}\ar[d]^-{\tilde f_2}&\zz/(F_N-1)\zz\ar[r]\ar[d]^-{=}&0\\
    0\ar[r]& \calF(\frp_N^{(r)}) \ar[r]^-{\fFN}&\Nnr\ar[rr]^-{(F-1)\times 1}&&\Nnr\ar[r]&\zz/(F_N-1)\zz\ar[r]&0
  }
\]
where
\[\begin{split}
&\delta_2(e_iz_1)=(u^mb-1)e_iz_0,\\
&\delta_2(e_iz_2)=(a-1)e_iz_0,\\
&\delta_2(e_iv_{j,k})=0,\\
&\tilde f_2(e_iz_0)=[(\theta_1)_i,1,\dots,1],\\
&\tilde f_3(e_iz_1)=[(\theta_1)_i,1,\dots,1],\\
&\tilde f_3(e_iz_2)=[\gamma_i,\dots,\gamma_i],\\
&\tilde f_3(e_iv_{k,j})=\fFN(\tfW(e_iv_{k,j})),\\
&\pi(e_iz_0)=e_i,
\end{split}\]
for all $i$, $j$ and $k$. Further, $X(2)=\ker(\delta_2|_{W'})$ and $\tilde f_4$ is the restriction of $\tilde f_3$ to $X(2) \oplus W$.
\end{lemma}

\begin{proof}
  We recall from \cite[Sec.~2]{Cobbe18} that the action of $\Gal(K_0/K) \times G$  on $\calI_{N/K}(\rhonr)$ is
  characterized by 
  \begin{eqnarray*}
    (F \times 1) [x_1, \ldots, x_d] &=& [U_Nx_d^{F_N}, x_2, \ldots, x_{d-1}], \\
    (F^{-n} \times \sigma) [x_1, \ldots, x_d] &=&
               [\rhonr(\tilde\sigma) x_1^{\tilde\sigma}, \ldots,  \rhonr(\tilde\sigma) x_d^{\tilde\sigma}], 
  \end{eqnarray*}
  where $x_i \in \prod_r \widehat{N_0^\times}$, the elements $F^{-n}$ and $\sigma \in G$ have the same restriction to $N \cap K_0$ and
  $\tilde\sigma \in \Gal(N_0/K)$ is uniquely defined by $\tilde\sigma|_{K_0} = F^{-n}$ and $\tilde\sigma|_{N} = \sigma$.
  Furthermore, we remark that the action of $G$ on $\calZ / (U_N - 1)\calZ$ is induced by
  \[
    a \cdot e_i = e_i, \quad b \cdot e_i = \rhonr(F^{-1}) e_i = U_K^{-1}e_i.
  \]
  The bottom sequence is exact by \cite[Thm.~3.3 and Lemma~2.1]{Cobbe18}. The proofs of the exactness of the top sequence
  as well as the proof of commutativity follow along the lines of proof in
  the one-dimensional case, see \cite[Lemma~7.4.1]{BC2}.
  For example, if we denote the coefficients of $U_K = u^m$ by $u_{ij}$, then
  \[
    \begin{split}\tilde f_2\circ\delta_2(e_iz_1)&=\tilde f_2((u^mb-1)e_i)=\tilde f_2\left(\sum_{j=1}^r bu_{j,i}e_j-e_i\right)\\
      &=\sum_{j=1}^r (1\times b)u_{j,i}\cdot [(\theta_1)_j,1,\dots,1]-[(\theta_1)_i,1,\dots,1] \\
      &=(1\times b)\cdot [(\theta_1^{u_{1,i}},\theta_1^{u_{2,i}},\dots,\theta_1^{u_{r,i}}),1,\dots,1]-[(\theta_1)_i,1,\dots,1]\\
      &=(1\times b)[u(\theta_1)_i,1,\dots,1]-[(\theta_1)_i,1,\dots,1]\\
      &=(F\times 1)(F^{-1}\times b)[u(\theta_1)_i,1,\dots,1]-[(\theta_1)_i,1,\dots,1]\\
      &=(F\times 1)[u^{-1}u(\theta_1)_i,1,\dots,1]-[(\theta_1)_i,1,\dots,1]\\
      &=((F-1)\times 1)\circ \tilde f_3(u_iz_1)
    \end{split}
  \]
\end{proof}

We will need  an explicit description of $X(2)$, which generalizes the one given in \cite[Lemma~7.4.3]{BC2}.

\begin{lemma}\label{X2generators}
We have
\[
    X(2) =\langle (a-1)e_iz_1-(u^mb-1)e_iz_2, \trace_a e_iz_2 \rangle_\ZpG.
\]
\end{lemma}

\begin{proof}
  The proof is just the $r$-dimensional analogue of that of \cite[Lemma~7.4.3]{BC2}.
\end{proof}

We let
\[
  f_4 \colon X(2) \oplus W \lra \calF(\frp_N^{(r)}) / \calF((\p_N^{p+1})^{(r)}).
\]
denote the composite of $\tilde f_4$ with the canonical projection. By Lemma \ref{f3surj} the  homomorphism $f_4$ is surjective.

In the next proposition, which is the analogue of \cite[Lemma~7.4.2]{BC2}, we will obtain an explicit
representative for the complex $M^\bullet(\calL)$, which we will use to compute the Euler characteristic
$\chi_{\Z_p[G], \BdR[G]}(M^\bullet(\calL), 0)$.

\begin{prop}\label{nicerepresentative}
The complex
\[
F^\bullet := [\ker f_4\lra W'\oplus W\lra\Z_p[G]^r z_0]
\]
with modules in degrees $0,1$ and $2$ is a representative of $M^\bullet(\calL)$ for $\calL = \p_N^{p+1}$.
\end{prop}

\begin{proof}
  If we recall the definition of $M^\bullet(\calL)$ from (\ref{Mbullet}), then it follows readily from
  Lemma \ref{deff4} and Lemma \ref{f3surj} that we have a quasi-isomorphism of complexes
  \[
    F^\bullet \lra M^\bullet(\calL).
  \]
\end{proof}

For the computation of the Euler characteristic of $M^\bullet(\calL)$ we continue to closely follow the approach of
\cite[Sec.~7.4]{BC2}. On these grounds we will only sketch the proofs, pointing out the parts which are specific
for the higher dimensional setting.

The next result is an analogue of \cite[Lemma~7.4.5]{BC2} and \cite[Lemma~7.4.6]{BC2}.
Recall that by assumption $(m,d) = 1$ and let $\tilde m$ denote an integer such that
\[
  m\tilde m \equiv 1 \pmod d.
\]

\begin{lemma}\label{deft1}
There exist $y_{i,1}\in W_{\geq 1}$ such that
\[\begin{split}t_{i,1} &:= (a-1)e_iz_1-(u^mb-1)e_iz_2+\\
&\qquad+\left(\sum_{j=2}^{m}\alpha_{j}(u^mb)^{1-(j-2)\tilde m}+\left(\alpha_1-\sum_{j=2}^{m}\alpha_{j}\right)(u^mb)^{\tilde m}\right)e_iw_0+ y_{i,1}\end{split}\]
and
\[t_{i,2}:=\T_ae_iz_2-\beta e_i w_{p-1}\text{ with }\beta=\begin{cases}\alpha_1&\text{if }m=1\\ \alpha_2&\text{if }m>1\end{cases}\]
are in the kernel of $f_4$. 
\end{lemma}

\begin{proof}
  Let $x_2\in\oo_{\Kur}$ be such that $x_2/\alpha_1\pmod{\p_{K'}}$ is a root of $X^p-X+A\theta_2$ and
  let $(1+x_2\theta_1)_i$ be the element in $\prod_r U_{N_0}^{(1)}$ whose $i$-th component is $1+x_2\theta_1$
  and all the other components are $1$. By the same proof as in \cite[Lemma~7.2.2]{BC2} we obtain
\[(1+x_2\theta_1)_i^{\rhonr(F_N)F_N-1}\equiv (1-\alpha_1\theta_1)_i\pmod{\p_{N_0}^2}.\]
Therefore, by \cite[Lemma~2.4]{Cobbe18} we may assume $\gamma_i\equiv (1+x_2\theta_1)_i\pmod{\p_{N_0}^2}$.

With this in mind, the proof of the lemma is the same as in \cite[Lemma~7.4.5]{BC2} and \cite[Lemma~7.4.6]{BC2}.
\end{proof}

\begin{lemma}
The elements
\[r_{i,1}=\T_a t_{i,1} + (u^mb-1)t_{i,2}\]
belong to $\ker f_4\cap W$ and their $\alpha_1w_0$-components are $(u^mb)^{\tilde m} \T_a e_i$.
\end{lemma}

\begin{proof}
Straightforward by the same calculations as in \cite[Lemma~4.2.5]{BC}.
\end{proof}

We redefine the matrix $\calM$ considered in the case of Hypothesis (I) taking into account these new modified elements $r_{i,1}$:

\[\calM=\begin{pmatrix}
0&0&\cdots&0&0&\epsilon u^{-m\tilde m}\epsilon^{-1} b^{-\tilde m}\\
-1&-1&\cdots&0&0&-\epsilon u^{-m\tilde m}\epsilon^{-1}b^{-\tilde m}\\
0&\!\!\!\epsilon u^{-m\tilde m}\epsilon^{-1}b^{-\tilde m}\!\!\!&\cdots&0&0&-\epsilon u^{-m\tilde m}\epsilon^{-1}b^{-\tilde m}\\
\vdots&\vdots&\ddots&\vdots&\vdots&\vdots\\
0&0&\cdots&\!\!\epsilon u^{-m\tilde m}\epsilon^{-1}b^{-\tilde m}\!\!\!\!\!\!\!&-1&-\epsilon u^{-m\tilde m}\epsilon^{-1}b^{-\tilde m}\\
0&0&\cdots&0&\epsilon u^{-m\tilde m}\epsilon^{-1}b^{-\tilde m}&-1\!-\!\epsilon u^{-m\tilde m}\epsilon^{-1} b^{-\tilde m}
\end{pmatrix}\!.\]

By an easy calculation $\det(\calM)=(-1)^{mr}(\det(u)^mb)^{-\tilde m(m-1)}$.

\begin{lemma}
The $r(pm+1)$ elements $t_{1,i}, t_{2,i}$, for $i=1,\dots,r$, $r_{i,k}$, for $i=1,\dots,r$, $k=2,\dots,m$ and $s_{i,j,k}$, for $i=1,\dots,r$, $j=0,\dots,p-2$, $k=1,\dots,m$ constitute a $\Z_p[G]$-basis of $\ker(f_4)$.
\end{lemma}

\begin{proof}
  It is enough to follow the proof of \cite[Lemma~7.4.9]{BC2}.
  Note that we can construct the matrix $\calM$ as in \cite{BC2} and that it also takes
  exactly the same shape up to the fact that all the entries are $r\times r$-blocks.
\end{proof}

By the same proof as in \cite{BC2} we can now reduce the computation of the term $\chi_{\Z_p[G], \BdR[G]}(M^\bullet(\calL), 0)$ to the determinant of a matrix $(w,\mm)$, which looks exactly as in \cite{BC2}, with the convention the elements in $\ZpG$ must be thought as diagonal $r\times r$ matrices. So in particular $m\times m$ blocks become $mr \times mr$ blocks and so on. Of course for the block $\calM$ we have to take the one defined above and not that of \cite{BC2}. With this in mind the matrix looks as follows:
\[\begin{pmatrix}
\frac{\sum_{i=0}^{d-1}(u^{m}b)^i}{u^{dm}-1}&(a-1)I_r&0&0&0&0&\cdots&0&0\\
0& \!\!\!1-u^mb\!\!\!&\T_a I_r&0&0&0&\cdots&0&0\\
0&v&0&\T_a\tilde I&(a-1)I_{rm}&0&\cdots&0&0\\
0&*&0&0&-I_{rm}&\!\!(a-1)I_{rm}\!\!&\cdots&0&0\\
0&*&0&0&*&-I_{rm}&\cdots&0&0\\
\vdots&\vdots&\vdots&\vdots&\vdots&\vdots&\ddots&\vdots&\vdots\\
0&*&0&0&*&*&\cdots&-I_{rm}&(a-1)I_{rm}\\
0&*&\calM_1&\tilde \calM&*&*&\cdots&*&-I_{rm}
\end{pmatrix}.
\]

Here $I_{rm}$ (resp. $I_r$) is the $rm\times rm$ (resp. $r\times r$) identity matrix, $\tilde I$ is obtained by $I_{rm}$ by removing the first $r$ columns, $v$ is an $r \times rm$ matrix, whose first $r$ rows coincide with the matrix $(u^mb)^{\tilde m}$ and $\frac{1}{u^{dm}-1}$ is the inverse of the matrix $u^{dm}-1$.

We are ready to record the result of the computation of the refined Euler characteristic of $M^\bullet(\calL)$. 

\begin{prop}\label{prop 118}
We assume the hypotheses (F) and (T) and let $\calL = \p_N^{p+1}$. Then the element 
$\chi_{\Z_p[G], \BdR[G]}(M^\bullet(\calL), 0)$ is contained in $K_0(\Z_p[G], \Q_p[G])$ and represented by
$\epsilon \in \Q_p[G]^\times$ where
\[
  \epsilon_{\chi\phi}^{}=
     \begin{cases}(-1)^{r}\frac{\det(u)^{m\tilde m}\phi(b)^{r\tilde m}p^{r m}}{\det(u^m\phi(b)-1)}&\text{if $\chi=\chi_0$}\\
                  (-1)^{(m-1)r}(\det(u)^m\phi(b)^r)^{-\tilde m(m-1)}(\chi(a)-1)^{r m(p-1)}&\text{if $\chi\neq\chi_0$}
                \end{cases}
\]
and $\chi\phi$ is the decomposition of a character in a ramified and an unramified component.
\end{prop}

\begin{proof}
The proof is a straightforward adaption of the computations in the proof of \cite[Prop.~7.4.10]{BC2}.
\end{proof}
\end{subsection}
\end{section}

\begin{section}{Rationality and functoriality}
From now on we set $\hat\partial^1 = \hat\partial_{\ZpG, \BdRG}^1$. As in \cite{BC2}, we consider the term
\[
  R_{N/K}=C_{N/K}+\Ucris+rm\hat\partial^1(t) - m \Utwist   -r U_{N/K}+\hat\partial^1(\epsilon_D(N/K,V)),\]
where $U_{N/K}$ is the unramified term defined by Breuning in \cite[Prop.~2.12]{Breuning04}. Recall that $R_{N/K}$ differs
  from \cite[(17)]{BC2} since we have to adapt the definition of $R_{N/K}$ as explained in Remark \ref{Remark Error}. 

\begin{prop}\label{RNK rational}
The element $R_{N/K}$ is rational, i.e. $R_{N/K}\in K_0(\ZpG, \Qp[G])$.
\end{prop}
\begin{proof}
    For elements $x,y \in K_0(\ZpG, \Qpc[G])$ we use the notation $x \equiv y$ when $x-y \in K_0(\ZpG, \QpG)$.
  By Propositions \ref{CNKgeneral}, \ref{eps coroll} and \ref{propUcris}  we get
  \[
    R_{N/K} \equiv r\hat\partial^1({\rho}) - rU_{N/K} - rT_{N/K},
  \]
  where $T_{N/K} := \hat\partial\left( \tau_\Qp(\Ind_{K/\Qp}(\chi))_{\chi \in \Irr(G)}\right)$
  is precisely the element defined by Breuning in \cite[Sec.~2.3]{Breuning04}.
  The result now follows, since this is exactly $r$ times the element obtained in \cite[Prop.~7.1.3]{BC2}.
\end{proof}

\begin{prop}\label{functoriality}
  Let $L$ be an intermediate field of $N/K$ and $H=\Gal(N/L)$. Let $\rho_H^G$ and $q_{G/H}^G$ be the restriction and quotient
  map from (\ref{K0 res}) and (\ref{K0 quot}), respectively. Then
\begin{enumerate}[(a)]
\item $\rho^G_H(R_{N/K})=R_{N/L}$.
\item If $H$ is normal in $G$, then $q^G_{G/H}(R_{N/K})=R_{L/K}$.
\end{enumerate}
\end{prop}

\begin{proof}
 By definition of the cohomological term  we have 
  \begin{eqnarray*}
    R_{N/K} &=& -\chi_{\ZpG, \BdRG}(M^\bullet, \exp_{V, N} \circ \comp_{V, N}^{-1}) +rm\hat\partial^1(t) - m \Utwist\\
    &\ &
    +U_{\mathrm{cris}, N/K} - r U_{N/K} + \hat\partial_{\ZpG, \BdRG}^1(\epsilon_D(N/K,V))
  \end{eqnarray*}
  with $M^\bullet = R\Gamma(N, T) \oplus \Ind_{N/\Qp}T[0]$. The functoriality properties of $rm\hat\partial^1(t)$ and $m\Utwist$ follow easily using the general formulas in \cite[Sec.~6.1 and 6.3]{BleyWilson}. Recalling also Lemma \ref{functUcris}, Lemma \ref{functepsilon} and \cite[Lemma~4.5]{BreuPhd}, 
  it remains to show
  \begin{equation}\label{funct a}
    \rho_H^G\left( \chi_{\ZpG, \BdRG}(M^\bullet, \exp_{V, N} \circ \comp_{V, N}^{-1}) \right)
    = \chi_{\Z_p[H], \BdR[H]}(M^\bullet, \exp_{V, N} \circ \comp_{V, N}^{-1})
  \end{equation}
  and
  \begin{equation}\label{funct b}
    q^G_{G/H}\left( \chi_{\ZpG, \BdRG}(M^\bullet, \exp_{V, N} \circ \comp_{V, N}^{-1}) \right)
    = \chi_{\Z_p[G/H], \BdR[G/H]}(M^\bullet, \exp_{V, L} \circ \comp_{V, L}^{-1}).
  \end{equation}
  The proof of (\ref{funct a}) follows along the same line of argument as the proof of \cite[Lemma 4.14 (1)]{BreuPhd}.
  We just have to replace $A = \mu_{p^n}$ in loc.cit. by $\calF[p^n]$.

  Since (\ref{funct b}) is essentially proved in the same way as part (2) of \cite[Lemma 4.14]{Breuning04}
  we only give a brief sketch.
  As in (17) of loc.cit. we obtain a canonical isomorphism
  \begin{equation}\label{funct eq 1}
    R\Gamma(L, T) \cong R\Hom_{\Z_p[H]}(\Zp, R\Gamma(N, T))
  \end{equation}
  in the derived category. For each $i$ the induced map from
  $H^i(\Qpc \tensor R\Gamma(L, T)) \cong \Qpc \tensor H^i(L, T)$
  to $H^i(\Qpc \tensor R\Hom_{\Z_p[H]}(\Zp, R\Gamma(N, T))) \cong \Qpc \tensor H^i(N, T)^H$
  is the map induced by the restriction $H^i(L, T) \lra H^i(N, T)$.

  For $i=1$ this restriction map is clearly given by the inclusion $\calF(\frp_L^{(r)}) \sseq \calF(\frp_N^{(r)})$ and for
  $i=2$ cohomology vanishes after tensoring with $\Qpc$ by our hypothesis (F).

  We further note that we have a canonical isomorphism
   \begin{equation}\label{funct eq 2}
     \left( \Ind_{N/\Qp}T \right)^H \cong  \Ind_{L/\Qp}T
   \end{equation}
   which is given by $x \mapsto x$, if we identify $ \Ind_{N/\Qp}T$ with
   the set of maps $x \colon G_\Qp \lra T$ satisfying $x(\tau\sigma) = \tau x(\sigma)$ for all $\tau \in G_N$ and
   $\sigma \in G_\Qp$.

   From (\ref{funct eq 1}) and (\ref{funct eq 2}) we derive a canonical isomorphism
   $M_L^\bullet \cong R\Hom_{\Z_p[H]}(\Zp, M_N^\bullet)$ and it finally remains to show that the induced maps (drawn as dotted
   arrows below) in the following diagram coincide with $\comp_{V,L}^{-1}$ and $\exp_{V, L}$, respectively.
   \[\xymatrix{
       \left( \Ind_{N/\Qp}(V)_{\BdR} \right)^H \ar[d]^\cong \ar[rr]^-{\comp_{V,N}^{-1}} &&
       \left( \DdRN(V)_\BdR \right)^H \ar[d]^\cong \ar[r]^-{\exp_{V,N}} &
       \left( H^1(N, V)_\BdR \right)^H =  \calF(\frp_N^{(r)})_\BdR^H  \ar[d]^\cong \\
       \Ind_{L/\Qp}(V)_\BdR  \ar@{..>}[rr]^-{\comp_{V,L}^{-1}} &&
       \DdR^L(V)_\BdR \ar@{..>}[r]^-{\exp_{V,L}} &
       H^1(L, V)_\BdR  =  \calF(\frp_L^{(r)})_\BdR
     }
   \]
   For $\exp_V$ this is immediate since it is defined as a connecting homomorphism (see, for example, \cite[p.~612]{BenBer})
   and thus compatible with restriction. Also for $\comp_V$ it follows from the definitions, see \cite[p.~625]{BenBer}.
 \end{proof}
\end{section} 

\begin{section}{Proof of the main results}
As in \cite{BC2} we also define
\[
  \tilde R_{N/K}=C_{N/K}+\Ucris+rm\hat\partial^1(t) - m\Utwist +\hat\partial^1(\epsilon_D(N/K,V)),
\]
so that $R_{N/K} =  \tilde R_{N/K} -  rU_{N/K}$.

We now argue as in the proof of \cite[Prop.~3.2.6]{BC2} (see page 359 of loc.cit.).
By Taylor's fixed point theorem together
with \cite[Prop.~2.12]{Breuning04} it can be shown that
$R_{N/K}=0$ in $K_0(\ZpG,\QpG)$ if and only if
$\tilde R_{N/K}=0$ in $K_0(\overline{\Z_p^{nr}}[G], \overline{\Q_p^{nr}}[G])$.

From Propositions \ref{CNKgeneral}, \ref{eps coroll} and \ref{propUcris} we conclude that
\begin{equation}\label{generic}
  \begin{split}\tilde R_{N/K}=&
    r\hat\partial^1(\rho_b)+r[\mathcal L,\id,\mathcal O_K[G]\cdot b]-
    \chi_{}(M^\bullet(\calL), 0) \\
    &+\hat\partial^1({}^*(\det(1-Fp^{-d_K}u^{-d_K})e_I))-\hat\partial^1({}^*(\det(1-u^{d_K}F^{-1})e_I))\\
    &+\hat\partial^1\left( \det(u)^{-d_K(s_K\chi(1) + m_\chi)}\right)_{\chi \in \Irr(G)} -
    r\hat\partial^1\left(\tau_\Qp( \Ind_{K/\Qp}(\chi) ) \right)_{\chi \in \Irr(G)}
  \end{split}
\end{equation}

  As in the one-dimensional case,
  see \cite[Proposition 3.2.6]{BC2}, the following three statement are equivalent in our present
  setting:
  \begin{itemize}
  \item[(a)] $\CEP$ is valid.
  \item[(b)] $R_{N/K} = 0$ in $K_0(\ZpG, \QpG)$.
  \item[(c)] $\tilde R_{N/K} = 0$ in $K_0(\overline{\Z_p^{nr}}[G], \overline{\Q_p^{nr}}[G])$.  
  \end{itemize}

\begin{proof}[{\bf Proof of Theorem \ref{intro thm tame}}]
  It suffices to show that $\tilde R_{N/K}=0$.
  From (\ref{generic}) together with Proposition \ref{tameH1H2} we obtain
  \begin{eqnarray*}
  \tilde R_{N/K} &=& 
                     r\hat\partial^1(\rho)+r[\frp_N, \id, \ON] + \hat\partial^1({}^*(\det(1-Fp^{-d_K}u^{-d_K})e_I)) \\
    &+&\hat\partial^1\left( \det(u)^{-d_K(s_K\chi(1) + m_\chi)}\right)_{\chi \in \Irr(G)} -
    r\hat\partial^1\left(\tau_\Qp( \Ind_{K/\Qp}(\chi) ) \right)_{\chi \in \Irr(G)}
\end{eqnarray*}
The term
\[
  \hat\partial^1\left(\left(\det(u)^{-d_K(s_K\chi(1) + m_\chi)}\right)_\chi\right)
\]
vanishes by \cite[Lemma~6.2]{IV}. Since $N/K$ is tame, the group ring idempotent $e_I$ is contained in $\ZpG$ and it is then easy to show that
\[
  {}^*(\det(p^{d_K}-Fu^{-d_K})e_I)\in\Zp[G]^\times.
\]
We conclude that
\[
  \begin{split}
    \hat\partial^1({}^*(\det(1-Fp^{-d_K}u^{-d_K})e_I))&=-r\hat\partial^1({}^*(p^{d_K}e_I)) \\
    &                                                   =-r[\mathfrak p_N,\id,\oo_N].
  \end{split}
\]
where the last equality follows from  \cite[(6.7)]{IV}. In conclusion, we have shown that
\begin{eqnarray*}
  \tilde R_{N/K} &=& 
                     r \left( \hat\partial^1(\rho) - \hat\partial^1\left(\tau_\Qp( \Ind_{K/\Qp}(\chi) ) \right)_{\chi \in \Irr(G)} \right).
\end{eqnarray*}
We finally conclude as in the proof of \cite[Thm.~3.6]{Breuning04} or as in \cite[page 517]{IV} to show that
\[
   \hat\partial^1(\rho) - \hat\partial^1\left(\tau_\Qp( \Ind_{K/\Qp}(\chi) ) \right)_{\chi \in \Irr(G)}  = 0.
\]
We recall that these proofs crucially use a fundamental result of M.~Taylor (see \cite[Thm.~31]{Froehlich83})
which computes the quotient of norm resolvents and Galois Gauss sums. We also note that the so-called non-ramified characteristic which
occurs in these results is an integral unit itself.
\end{proof}

\begin{proof}[{\bf Proof of Theorem \ref{intro thm weak}}]
  The crucial input in our proof is a result of Picket and Vinatier in \cite{PickettVinatier}.
  In \cite[Sec.~5.1]{BC} we used this
  result to construct an integral normal basis generator $b = p^2\alpha_M\theta_2$ of $\calL = \frp_N^{p+1}$. Hence the
  expression in (\ref{generic}) simplifies to
  \begin{equation}\label{generic 2}
  \begin{split}\tilde R_{N/K}=&
    r\hat\partial^1(\rho)-
    \chi_{\Z_p[G], \BdR[G]}(M^\bullet(\calL), 0) \\
    &+\hat\partial^1({}^*(\det(1-Fp^{-d_K}u^{-d_K})e_I))-\hat\partial^1({}^*(\det(1-u^{d_K}F^{-1})e_I))\\
    &+\hat\partial^1\left( \det(u)^{-d_K(s_K\chi(1) + m_\chi)}\right)_{\chi \in \Irr(G)} -
    r\hat\partial^1\left(\tau_\Qp( \Ind_{K/\Qp}(\chi) ) \right)_{\chi \in \Irr(G)}
  \end{split}
\end{equation}
By \cite[Prop.~5.2.1]{BC} the element
$\hat\partial^1(\rho) - \hat\partial^1\left(\tau_\Qp( \Ind_{K/\Qp}(\chi) ) \right)_{\chi \in \Irr(G)}$ is represented by
{$\hat\partial^1(\eta^{-1})$} with $\eta$ as in \cite[Prop.~5.2.1]{BC}. Furthermore,
\[
  \hat\partial^1\left( \det(u)^{-d_K(s_K\chi(1) + m_\chi)}\right)_{\chi \in \Irr(G)} =
  \hat\partial^1\left( \det(u)^{-mm_\chi)}\right)_{\chi \in \Irr(G)}.
\]

{From now on, we have to distinguish the three conditions in the statement of the Theorem. Let us start with (a), i.e. Hypothesis (I). In this case} $\tilde R_{N/K}$ is represented by
\[\begin{split}
\tilde r_{\chi\phi}&=
\begin{cases}\frac{\frd_K^{r/2}\norm_{K/\Qp}(\theta_2|\phi)^r p^{2rm}\det(1-p^{-m}u^{-m}\phi(b^{-1}))}{p^{rm}\det(1-u^m\phi(b))}&\text{if $\chi=\chi_0$}\\
\frac{\frd_K^{r/2}\norm_{K/\Qp}(\theta_2|\phi)^r p^{rm}\chi(4)^r\phi(b)^{-2r} \det(u)^{-2m}}{(-1)^{r(m-1)}(u^{-m}\phi(b)^{-1}-1)(\chi(a)-1)^{rm(p-1)}}&\text{if $\chi\neq\chi_0$}\end{cases}\\
&=\begin{cases}(-1)^r\frac{\frd_K^{r/2}\norm_{K/\Qp}(\theta_2|\phi)^r}{\det(1-u^m\phi(b))}\det(u)^{-m}\phi(b)^{-r}\det(1-p^mu^{m}\phi(b))&\text{if $\chi=\chi_0$}\\
(-1)^r\frac{\frd_K^{r/2}\norm_{K/\Qp}(\theta_2|\phi)^r}{\det(1-u^m\phi(b))}\det(u)^{-m}\phi(b)^{-r}\left(\frac{-p}{(\chi(a)-1)^{p-1}}\right)^{rm}\chi(4)^r&\text{if $\chi\neq\chi_0$.}\end{cases}
\end{split}\]

Let $W_{\theta_2}\in\mathcal O_p^t[G]\subset \overline{\Zpnr}[G]^\times$ be defined by
\[\chi\phi(W_{\theta_2})=(-1)^{r}\frd_K^{r/2}\norm_{K/\Qp}(\theta_2|\phi)^r.\]

So
\[\tilde r=\frac{W_{\theta_2}\det(u)^{-m}b^{-r}}{\det(1-u^m b)} \left(\det(1-p^mu^mb)e_a+(-\tilde u)^m\sigma_4(1-e_a)\right),\]
where $\tilde u\in\Z_p[a]$ is a unit whose augmentation is congruent to $1$ modulo $p$, as in \cite[Sect.~8]{BC2}. Then the proof works as in \cite{BC2}.

Let us now consider case (b), i.e. Hypothesis (T). In this case $\tilde R_{N/K}$ is represented by

\[\begin{split}
\tilde r_{\chi\phi}&=
\begin{cases}(-1)^{r}\frac{\frd_K^{r/2}\norm_{K/\Qp}(\theta_2|\phi)^r p^{2rm}\det(u^m\phi(b)-1)\det(1-p^{-m}u^{-m}\phi(b^{-1}))}{\det(u)^{m\tilde m}\phi(b)^{r\tilde m}p^{rm}\det(1-u^m\phi(b))}&\text{if $\chi=\chi_0$}\\
(-1)^{r(m+1)}\frac{\frd_K^{r/2}\norm_{K/\Qp}(\theta_2|\phi)^r p^{rm}\chi(4)^r\phi(b)^{-2r} \det(u)^{-2m}}{(\det(u)^m\phi(b)^r)^{-\tilde m(m-1)}(\chi(a)-1)^{rm(p-1)}}&\text{if $\chi\neq\chi_0$}\end{cases}\\
&=\begin{cases}\chi\phi(W_{\theta_2})\phi(b)^{-r-r\tilde m}\det(u)^{-m-m\tilde m}\det(1-p^m u^{m}\phi(b))&\!\text{if $\chi=\chi_0$}\\
\chi\phi(W_{\theta_2})\phi(b)^{-r-r\tilde m}\det(u)^{-m-m\tilde m}\left(\frac{-p}{(\chi(a)-1)^{p-1}}\right)^{rm}\chi(4)\det(u)^{m^2\tilde m-m}&\!\text{if $\chi\neq\chi_0$,}
\end{cases}
\end{split}\]

So
\[\tilde r=W_{\theta_2}b^{r-r\tilde m}\det(u)^{-m-m\tilde m}\left(\det(1-p^mu^mb)e_a+(-\tilde u)^m\sigma_4 \det(u)^{m^2\tilde m-m}(1-e_a)\right).\]
As in \cite{BC2}, Hypothesis (T) implies that $\det(u)^{m^2\tilde m-m}\equiv 1\pmod{1-\zeta_p}$. Then the proof works as in \cite{BC2}, using the computations of the case of Hypothesis (I).

{It remains to consider the case (c). }
  Let $E/K'$ be the unramified extension of degree $\tilde d$. By the functoriality result of Proposition \ref{functoriality} (a) it is enough
  to show $R_{NE/K}=0$, which is true since for $NE/K$ we can apply part (b) of the theorem.
\end{proof}

We finally prove Theorems \ref{intro thm tame geom} and \ref{intro thm weak geom}. For that purpose we recall the following lemma.

\begin{lemma}\label{Mazur lemma}
  Let $A$ be an $r$-dimensional abelian variety defined over the $p$-adic number field $L$ with good ordinary reduction and
  let $U$ be the twist matrix of $A$. Then $\det(U - 1) \ne 0$.
\end{lemma}

\begin{proof}
  This is shown in  \cite[Cor.~4.38]{Mazur72} or in the proof of \cite[Thm.~2]{LubinRosen}.
\end{proof}

\begin{proof}[\bf Proof of Theorems \ref{intro thm tame geom} and \ref{intro thm weak geom}]
The results are immediate from Theorems \ref{intro thm tame} and \ref{intro thm weak}. Note that Hypothesis (F) and the condition $\det(\rhonr(F_N)^{\tilde d} - 1) \ne 0$ are automatically satisfied by Lemma \ref{Mazur lemma}.
\end{proof}

\end{section}

\end{document}